\documentclass[12pt]{amsart}

\usepackage[utf8]{inputenc}
\usepackage[english]{babel}

\usepackage{latexsym,color,amsmath,amsthm,amssymb,amscd,amsfonts}
\usepackage{tikz}
\usetikzlibrary{arrows}
\usepackage{scalefnt}
\usepackage[font=small,labelfont=bf]{caption}
\usepackage{float}
\usepackage{graphicx}
\usepackage{relsize}
\usepackage{amsopn}
\usepackage{mathrsfs}  
\usepackage[hidelinks]{hyperref}
\usepackage{bm}
\usepackage{bbm}
\usepackage[shortlabels]{enumitem}
\usepackage{multicol}
\usepackage{float} 
\usepackage{todonotes}
\usepackage{cite}
\usepackage{soul}

\usepackage[margin=2.5cm]{geometry}

\newtheoremstyle{nonum}{}{}{\itshape}{}{\bfseries}{.}{ }{\thmnote{#3}}

\newtheorem{theorem}{Theorem}[section]
\newtheorem*{theorem*}{Theorem}

\newtheorem{corollary}[theorem]{Corollary}
\newtheorem{lemma}[theorem]{Lemma}

\newtheorem{proposition}[theorem]{Proposition}

\newtheorem{example}[theorem]{Example}
\newtheorem{definition}[theorem]{Definition}
\newtheorem*{definition*}{Definition}
\newtheorem{remark}[theorem]{Remark}
\newtheorem*{remark*}{Remark}
\newtheorem*{question*}{Question}

\theoremstyle{nonum}

\title{Homogeneous maximizers of the Blaschke--Santal{\'o}-type functionals}

\author{Alexander V. Kolesnikov}
\address{HSE University, Russian Federation}
\email{sascha77@mail.ru}

\begin{document}

\maketitle

\begin{abstract}
    We study Blaschke--Santal{\'o}-type inequalities for $N \ge 2$ sets (functions) and a special class of cost functions. In particular, we prove new results about the reduction of the maximization problem for  the Blaschke--Santal{\'o}-type functional to the homogeneous case (functional inequalities on the sphere) and prove new inequalities for $N>2$ sets by the symmetrization argument.
    We also discuss the links to the multimagrinal optimal transportation problem and the related sharp transportation-information inequalities.
\end{abstract}

\section{Introduction}

In this work, we discuss various extensions of the classical Blaschke--Santal\'o inequality
\begin{equation}
    \label{BSsets}
|A||A^\circ| \le |B_2|^2
\end{equation}
and its functional counterpart
\begin{equation}
\label{BSfunctions}
\int_{\mathbb{R}^n} e^{-V}dx  \int_{\mathbb{R}^n} e^{-V^*}dy \le \Bigl( \int_{\mathbb{R}^n} e^{-\frac{|x|^2}{2}}dx\Bigr)^2.
\end{equation}
Here $A \subset \mathbb{R}^n$ is a convex symmetric ($-A=A$) body, and $V \colon \mathbb{R}^n \to (-\infty,\infty]$ is a proper convex even function, $$A^{\circ}=\{y: \langle x,y \rangle \le 1\}$$ is the polar set of $A$, and $$V^*(y) = \sup_x (\langle x,y \rangle - V(x))$$ is the Legendre transform of $V$. Finally, $B_p = \{ x: \sum_{j=1}^n |x_j|^p \le 1\} \subset \mathbb{R}^n$ is the $l^p$-unit ball, and $|\cdot|$ is the Lebesgue volume.
The inequality (\ref{BSsets}) was proved by  Santal\'o \cite{santalo} (and by Blaschke in dimension 3). The inequality (\ref{BSfunctions}) is substantially younger and was first proved by 
K. Ball \cite{Ball1986} in 1986.

Let us consider $N \ge 2$ copies of $\mathbb{R}^n$.
The $i$-th copy will be equipped with the standard coordinate system $x_i$. We write
$$
x = (x_1, \cdots,x_N) \in (\mathbb{R}^n)^N, \ x_i \in \mathbb{R}^n.
$$ The coordinate representation of an arbitrary $x_i \in \mathbb{R}^n$ will be written in the following way:
$$
x_i = (x_{i,1}, \cdots, x_{i,j}, \cdots, x_{i,n}).
$$
The following interesting example of the Blaschke--Santal{\'o} inequality was proved by Kalantzopoulos and Saroglou in \cite{KS}

\begin{theorem} (Kalantzopoulos, Saroglou, \cite{KS})
\label{example1}
Let  $N \ge 2$ be a natural number. Consider the following cost function $$
c(x) = c(x_1,\cdots,x_N) 
= \sum_{j=1}^n x_{1,j}x_{2,j} \cdots x_{N,j} 
$$ 
on $(\mathbb{R}^n)^N$.
Assume we are given $N$ symmetric convex sets
 $A_i \subset \mathbb{R}^n$, $1 \le i \le N$  such that
$c(x_1,\cdots,x_N) \le 1$ on the set $\prod_{i=1}^N A_i \subset (\mathbb{R}^n)^N$. 
Then
$$
\prod_{i=1}^N |A_i| \le |B_N|^N.
$$
Let $\{V_i\}$, $1 \le i \le N,$ be even measurable functions with values in $(-\infty,+\infty]$ and satisfying
$
\sum_{i=1}^N V_i(x_i) \ge c(x).
$
Assume, in addition, that $e^{-V_i}\in L^1(dx_i)$ for every $i$. 
Then
$$
\prod_{i=1}^N \int_{\mathbb{R}^n} e^{-V_i(x_i)} dx_i 
\le \Bigl( \int_{\mathbb{R}^n} e^{-\frac{1}{N} \sum_{i=1}^n |y_i|^N} dy\Bigr)^N.
$$
\end{theorem}

We  extend this result in section 6 of our work, where we obtain, in particular, a generalization for functionals of the type $\prod_{i=1}^N |A_i|^{\frac{1}{\alpha_i}}$.
This example can be considered an illustration of several results of this work, where the central question is the relation between the functional and the set version of the Blaschke--Santal\'o type functional.  
We can put it into the following form: {\it when does the Blaschke--Santal{\'o} functional admit homogeneous extremizers for  given homogeneous reference measures and a homogeneous cost function? }

Before we  explain the  results of this paper, 
let us provide an incomplete overview of the related previous results  below. See more information in the nice recent survey of Fradelizi, Meyer, and Zvavitch
\cite{FMZ}.

\begin{itemize}
    \item Generalizations to the  non-symmetric case, proofs of the classical case.

It was proved in Artstein--Klartag--Milman \cite{AAKM} that for any convex (not necessary symmetric) $V$ there exists a point $a \in \mathbb{R}^n$ such that 
    \begin{equation}
\label{BSfunctions-nonsym}
\int_{\mathbb{R}^n} e^{-\tilde{V}}dx  \int_{\mathbb{R}^n} e^{-\tilde{V}^*}dy \le \Bigl( \int_{\mathbb{R}^n} e^{-\frac{|x|^2}{2}}dx\Bigr)^2,
\end{equation}
where $\tilde{V}(x) = V(x-a)$.
If $\int_{\mathbb R^n} xe^{-V}dx=0$, one can take $a=0$.

Let us give a short comment about the proofs in the classical case. 
A short proof of (\ref{BSfunctions}) not relying on the set version (\ref{BSsets}) was obtained by Lehec \cite{Lehec}.  
Bianchi and Kelli \cite{BK} proved the classical Blaschke--Santal\'o inequality using the Fourier analysis.
Nakamura and Tsuji \cite{NT} obtained a proof of the functional  Blaschke--Santal\'o inequality based on the heat flow approach (see a non-symmetric generalization in Cordero-Erausquin--Fradelizi--Langharst \cite{CoFrLa} and a more direct proof in Cordero-Erausquin--Gozlan--Nakamura--Tsuji \cite{CEGoNT}). Moreover, they obtained a more general functional $L^p-L^q$ inequality. 
A variational proof  of (\ref{BSfunctions}) based on the uniqueness of the solution to the corresponding Monge--Amp\`ere equation has been obtained in Colesanti--Kolesnikov--Livschyts--Rotem \cite{CKLR} (see also subsection 3.1 here).
For the  questions of stability see Barthe--B\"or\"oczky--Fradelizi\cite{BBF} and the reference therein.

 \item Generalized Legendre-type transforms.

 Some other types of the dualities for functions  and the related Blaschke--Santal\'o-type inequalities have been considered already in \cite{AAKM}.  
 An important extension, that is particularly relevant to our work is the result of Fradelizi and Meyer \cite{FradeliziMeyer2008(2)}. They established an inequality of the type
 $$
\int_{\mathbb{R}^n} f(x) dx \int_{\mathbb{R}^n} g(y) dy \le \Bigl( \int_{\mathbb{R}^n} \rho(|x|^2) dx\Bigr)^2
 $$
 under the constraint $f(x) g(y) \le \rho(\langle x- z, y - z \rangle)$, where $f,g \colon \mathbb{R}^n \to \mathbb{R}_+$ and $\rho \colon \mathbb{R}_+ \to \mathbb{R}_+$ (see \cite{FradeliziMeyer2008(2)} for precise formulation).
Another result of \cite{FradeliziMeyer2008(2)} is the following generalization of (\ref{BSfunctions}):
\begin{equation}
\label{BS-FM}
\int_{\mathbb{R}^n} \rho(V)dx \int_{\mathbb{R}^n} \rho(V^*) dy \le \Bigl(\int_{\mathbb{R}^n} \rho\bigl({|x|^2}/{2}\bigr) dx\Bigr)^2.
\end{equation}
Here $\rho$ is non-increasing and log-concave, $V$ is even and convex.

There are several extensions of the Blaschke--Santal{\'o} inequality for special types of the generalized Legendre transform. See, for instance, Rotem  \cite{Rotem2014},   
Artstein--Slomka \cite{ArtsteinSlomka},
Artstein--Sadovski--Wyczesany \cite{ASW}, 
Fradelizi--Gozlan--Sadovsky--Zugmeyer \cite{FGSZ}.

\item Relation to the optimal transportation and the information theory.

It is well--known that (\ref{BSfunctions}) admits a transportation counterpart. This can be established by the arguments going back to B.~Maurey (see M.~Fathi \cite{Fathi}): let $\gamma$ be the standard Gaussian measure and $f \cdot \gamma, g \cdot \gamma$ be probability measures such that $f \cdot \gamma$ is centered: $\int_{\mathbb{R}^n} xf d \gamma=0$. Then
the following transportation inequality holds:
\begin{equation}
    \label{BStransport}
\frac{1}{2} W^2_2(f \cdot \gamma, g \cdot \gamma) 
\le \int_{\mathbb{R}^n} f \log f d\gamma + \int_{\mathbb{R}^n} g \log g d\gamma.
\end{equation}
Here $W_2$ is the transportation distance on probability measures 
$$
W^2_2(\mu_1,\mu_2) = \inf_{\pi \in \Pi(\mu_1,\mu_2)} \int_{(\mathbb{R}^n)^2} |x-y|^2 \pi(dx dy),
$$
where $ \Pi(\mu_1,\mu_2)$ is the set of probability measures with projections $\mu_1,\mu_2$.
Inequality (\ref{BStransport}) can be extracted from (\ref{BSfunctions})  and vice versa by relatively simple arguments.
Later Gozlan \cite{Gozlan} extended the arguments of Fathi and established a transportation  version of the famous 
Mahler conjecture. Following observation of Kolesnikov \cite{Kolesnikov2018}  about transportation functionals on the sphere 
Fradelizi--Gozlan--Sadovsky--Zugmeyer \cite{FGSZ} proved a variant of the inequality 
(\ref{BStransport}) on $\mathbb{S}^{n-1}$ for an appropriate cost function (see also \cite{BKS}). Courtade--Fathi--Mikulincer obtained in \cite{CFM} a stochastic proof of (\ref{BStransport}).

\item Generalization for $N>2$ functions.

Motivated by the theory of geodesic barycenters ("Wasserstein barycenters")  of probability measures, 
Kolesnikov and Werner \cite{KW} conjectured that the following inequality  might be true:
\begin{equation}
\label{multBSf}
    \prod_{i=1}^N \int_{\mathbb{R}^n} e^{-V_i(x_i)} dx_i \le \Bigl(\int_{\mathbb{R}^n} e^{-\frac{|x|^2}{2} } dx\Bigr)^N
\end{equation}
under assumption that the functions $\{V_i(x_i)\}$ are even and  satisfy 
$$
\sum_{i=1}^N V_i(x_i) \ge \frac{1}{N-1}\sum_{j \ne i} \langle x_i,x_j \rangle.
$$
The choice of the cost function
$$
c_{bar} = \sum_{j \ne i} \langle x_i,x_j \rangle
$$
is motivated by the barycenter problem.
Clearly, (\ref{multBSf}) is a generalization of (\ref{BSfunctions}).

Kolesnikov and Werner \cite{KW} proved that (\ref{multBSf}) holds under assumption of unconditionality. Some generalizations of this result were obtained by Kalantzopoulos and Saroglou \cite{KS}. In particular, the following  version of (\ref{multBSf}) was conjectured:
\begin{equation}
    \label{multBSsets}
\prod_{i=1}^N |A_i| \le |B_2|^N,
\end{equation}
where $A_i$ are convex symmetric sets satisfying
$\sum_{j \ne i} \langle x_i,x_j \rangle \le \frac{N(N-1)}{2}$ for any tuple $\{x_i\}$ such that $x_i \in A_i$.
The question about inequality (\ref{multBSsets}) remains open to our best knowledge.

Inequality (\ref{multBSf}) was proved in full generality by   Nakamura and Tsuji \cite{NT2}. Finally, Courtade and Wang \cite{CW}  obtained a generalization of (\ref{multBSf}) in the non-symmetric case.

\item Weighted Blaschke--Santal\'o inequality.

Another generalization of (\ref{BSfunctions}) was obtained in Colesanti--Kolesnikov--Livschyts--Rotem \cite{CKLR}.
They studied inequalities of
the following type:
\begin{equation}
\label{BSPhi}
\Bigl( \int_{\mathbb{R}^n} e^{-V} dx\Bigr)^{\frac{1}{p}} 
\Bigl( \int_{\mathbb{R}^n} e^{-\frac{1}{p-1}V^*(\nabla \Phi)} dx\Bigr)^{\frac{p-1}{p}}
\le \int_{\mathbb{R}^n} e^{-\Phi} dx,
\end{equation}
where $\Phi$ is a $p$-homogeneous convex even function and $p>1$.

Assume we are given a strictly convex (and sufficiently regular) $p$-homogeneous even function $\Phi$ such that  (\ref{BSPhi}) holds for every even $V$. 
$p$-homogeneity implies that $\Phi = \frac{1}{1-p} \Phi^*(\nabla \Phi)$, hence $\Phi$
is a maximum point of the functional
\begin{align*}
\mathcal{F}_{p,\Phi}(V) & = \Bigl( \int_{\mathbb{R}^n} e^{-V} dx\Bigr)^{\frac{1}{p}} 
\Bigl( \int_{\mathbb{R}^n} e^{-\frac{1}{p-1}V^*(\nabla \Phi)} dx\Bigr)^{\frac{1}{q}}
\\& = \Bigl( \int_{\mathbb{R}^n} e^{-V} dx\Bigr)^{\frac{1}{p}} 
\Bigl( \int_{\mathbb{R}^n} e^{-\frac{1}
{p-1}V^*} \det D^2 \Phi^*dy\Bigr)^{\frac{1}{q}}.
\end{align*}
We observe that in general  the maximizer
$\Phi$ corresponds to the non-Gaussian distribution $\frac{e^{-\Phi}}{\int e^{-\Phi}dx} dx$.

It was realized in \cite{CKLR} that inequality (\ref{BSPhi}) can indeed hold for some functions $\Phi$. The simplest non-Gaussian example is given by
$$
\Phi = \sum_{i=1}^n |x_i|^p
$$
with $p>2$. On the other hand, the same function with $p<2$ does not satisfy   (\ref{BSPhi}).

\item Relation to other problems.

Finally, we mention several interesting developments related to the inequalities of the Blaschke--Santal{\'o} type.
First, motivated by the Nakamura--Tsuji result \cite{NT2},
E.~Milman \cite{Milman} obtained a new proof of the Gaussian correlation inequality that was a long-standing conjecture.  Nakamura and Tsuji \cite{NT3} generalized his proof and obtained a non-symmetric version of the Gaussian correlation inequality. 
Another interesting development concerns relation of the Mahler conjecture  to the Bourgain slicing problem, see 
Klartag \cite{Klartag}, Fradelizi--Sola \cite{FS}, and 
Mastrantonis--Rubinstein \cite{BMR}.
 Kolesnikov--Livschyts--Rotem studied in \cite{KLR} the infinitesimal versions of (\ref{BSPhi}) in relation to the log-Brunn--Minkowski problem.

\end{itemize}

Motivated by these results, we  study the following generalized Blaschke--Santal\'o functional:

\begin{equation}\label{BSF0}
\mathcal{BS}_{\alpha,m}(V_1,\cdots,V_N)=
\prod_{i=1}^N \Bigl( \int_{\mathbb{R}^n} e^{-\alpha_i V_i(x_i)}  m_i(dx_i) \Bigr)^{\frac{1}{\alpha_i}}.
\end{equation}
Here $m_i$ are the reference measures on $\mathbb{R}^n$, and  $\alpha_i$
are given positive numbers.

As in the classical case, we always assume that the functions $V_i$ satisfy 
\begin{equation}
\label{admiss}
\sum_{i=1}^N V_i(x_i) \ge c(x_1,\cdots,x_N)=c(x)
\end{equation}
for a given cost function $c$.
If (\ref{admiss}) holds, we say that the tuple $\{V_i\}$ is {\bf admissible}.

In this work, we study, in particular, the following  questions.

\begin{itemize}
    \item[A)] Equivalence between the set and the functional version.

The equivalence of the set inequality (\ref{BSsets}) and the functional inequality (\ref{BSfunctions})
is well known. Implication (\ref{BSfunctions}) $\Longrightarrow$ (\ref{BSsets}) is almost immediate (one takes $V = \frac{1}{2} |x|^2_K, V^* = \frac{1}{2} |x|^2_{K^\circ}$, where $|\cdot |_K$ is the Minkowski functional of $K$ and applies the layer-cake decomposition).
Implication (\ref{BSsets}) $\Longrightarrow$ (\ref{BSfunctions}) is trickier, the reader can find it, for instance, in \cite{AAKM}, \cite{Ball1986}.

For the case $N>2$ and the barycentric cost function $c_{bar}$, the equivalence between (\ref{multBSf}) and (\ref{multBSsets}) is not known.

\item[B)] Gaussianity and homogeneity  of the maximizers.

A remarkable observation of Nakamura and Tsuji was the fact that the inequality (\ref{multBSf}) fits into a large family of inequalities possessing  Gaussian maximizers (the so-called Gaussian saturation principle).  Inequalities of such type are called Brascamp--Lieb inequalities and they go back to the classical work of Brascamp and Lieb  \cite{BL}.

A relevant observation from Colesanti--Kolesnikov--Livschyts--Rotem \cite{CKLR} was the {\bf homogeneity} of the maximizers of $\mathcal{F}_{p,\Phi}$. This property has a certain correspondence to the Gaussian saturation  principle for inequalities of the Brascamb--Lieb type, but seems to be even more general, because it can potentially be  applied to situations, where Gaussian maximizers are not possible. An example is given by the inequality 
(\ref{BSPhi}) for  $
\Phi = \sum_{i=1}^n |x_i|^p
, p>2$,
proved in \cite{CKLR}.

\item[C)] Optimal transportation counterpart for the Blaschke--Santal\'o-type inequalities for  $\mathcal{F}_{p,\Phi}$.

    We are interested in generalizations of the sharp  Talagrand inequality (\ref{BStransport}) for the case of $N>2$ functions and non-quadratic costs. 
Some developments in this direction have been obtained in  
 \cite{Fathi}, \cite{KW}, \cite{Gozlan}, \cite{FGSZ}, \cite{CFM}. This part is strongly related 
to the multimarginal and barycentric transportation problems, in particular, to duality theory and multiple Legendre-type transforms.  
\end{itemize}

We want to find sufficient conditions ensuring  that all the maximizers of (\ref{BSF0}) are homogeneous.
First, we assume that 
the reference measures 
    $m_i = \rho_i(x_i)dx_i, \ 1 \le i \le N,$
 admit  $r_i$-homogeneous densities
   \begin{equation}
   \label{dens-hom}
\rho_i(tx_i) = t^{r_i} \rho_i(x_i),
\end{equation}
where  $r_i$ are some nonnegative numbers.

Of course, another natural assumption is that the cost function $c$ is homogeneous as well. The minimal requirement for $c$ will be $p$-homogeneity with respect to $x=(x_1,\cdots,x_N)$:
\label{cost-hom}
\begin{equation}
c(tx) 
= t^{p} c(x)
\end{equation}
for some positive number $p>0$.

Assume that the maximizers $\{V_i(x_i)\}$ of  (\ref{BSF0}) are $\beta_i$-homogeneous
\begin{equation}
    \label{potential0hom}
V_i(tx_i) = t^{\beta_i} V_i(x_i), \ \ \ t \ge 0.
\end{equation}
First, we observe a natural consistency condition for numbers $\alpha_i, \beta_i,r_i,p$.
We observe that $\beta_i$-homogeneity of $V_i$ and $p$-homogeneity of $c$ immediately imply that for every $t>0$ one has
$$
\sum_{i=1}^N V_i(t^{1-\frac{p}{\beta_i}}x_i) \ge c(x).
$$
Thus, the tuple $\{V_{i,t}\}$, where $V_{i,t}(x_i) = V_i(t^{1-\frac{p}{\beta_i}}x_i)$, is admissible.
Next, we compute using $r_i$-homogeneity of $\rho_i$ and the change of variables:
$$
\mathcal{BS}_{\alpha,m}(V_{1,t},\cdots,V_{N,t}) = t^{\sum_{i=1}^N \frac{n+r_i}{\alpha_i}\bigl( \frac{p}{\beta_i}-1\bigr) } \mathcal{BS}_{\alpha,m}(V_{1},\cdots,V_{N}).
$$
From this relation, we see that the tuple $\{V_i\}$ cannot be a maximizer unless one has 
$
\sum_{i=1}^N \frac{n+r_i}{\alpha_i}\bigl( \frac{p}{\beta_i}-1\bigr) =0.
$
Thus, it is necessary to assume the following consistency condition:
\begin{equation}
\label{consist1}
p = \frac{
\sum_{i=1}^N \frac{n+r_i}{\alpha_i }}{
\sum_{i=1}^N \frac{n+r_i}{\alpha_i \beta_i}}.
\end{equation}

We note that (\ref{consist1}) is a minimal requirement for the
existence of maximizers, but it is not sufficient (see Example 5.15 in \cite{CKLR}). 
Another observation is that  for many homogeneous cost functions, there exists a limited choice of $\beta_i$ such that the set of admissible $\beta_i$-homogeneous tuples $\{V_i\}$ is not empty.
Consider, for instance, the barycentric function
$
c_{bar}(x) = \frac{1}{N-1} \sum_{i,j=1}^N \langle x_i, x_j \rangle.
$
The reader can easily verify that in this case, the inequality 
$\sum_{i=1}^N V_i(x_i) \ge c(x)$ is possible for $\beta_i$-homogeneous $V_i$ only if $\beta_i=2$.

This is why, in addition to (\ref{consist1}),
we will assume another  consistency condition:
\begin{equation}
\label{consist2}
\frac{n+r_1}{\beta_1} = \frac{n+r_2}{\beta_2} =\cdots = \frac{n+r_N}{\beta_N}. 
\end{equation}
Under  (\ref{consist2})  the condition (\ref{consist1}) can be equivalently rewritten in the following way:
\begin{equation}
\label{cons-p}
p = \frac{\sum_{i=1}^{N} \frac{\beta_i}{\alpha_i}}{\sum_{i=1}^{N} \frac{1}{\alpha_i}}. 
\end{equation}

In this work, we extend several results on  different forms of the Blaschke--Santal{\'o} inequality  to the following class of cost functions:

\begin{definition}
    We say that a function $c(x_1,\cdots,x_N)$
    is  $(p_1,\cdots,p_N)$-homogeneous if for all $1 \le i \le N$ and $t\ge 0$ one has
$$
c(x_1,\cdots,t x_i, \cdots,x_N) 
= t^{p_i} c(x).
$$
\end{definition}

The central result of this work is the following Theorem on the homogeneity of the maximizers of the Blaschke--Santal{\'o} functional:

\begin{theorem}
\label{homogtheorem0}
Let $m_i$ satisfy assumption (\ref{dens-hom}) and have densities $\rho_i$ that are  positive almost everywhere. 
Let $c$ be  a cost function with the following properties: 
\begin{itemize}
\item 
$c$ is a continuous  $(p_1,\cdots,p_N)$-homogeneous for some numbers $p_i>0$
\item for every 
$1 \le i \le N$, there exist numbers 
$s_j \in \{-1,1\}$, $j \ne i$ such that 
$$
c(x_1, \cdots, x_{i-1}, - x_i, x_{i+1},\cdots,x_N) = 
c(s_1x_1, \cdots,  s_{i-1}x_{i-1}, x_i,  s_{i+1}x_{i+1},\cdots, s_N x_N) 
$$
\item 
  $c$ is   convex with respect to every variable, i.e., $$
x \to c(x_1, \cdots, x_{i-1},x, x_{i+1}, x_N)
$$ is a convex function for all $i \in [1,\cdots,N]$.
\end{itemize}
 
Consider the functional 
$$
\mathcal{BS}_{\alpha,m}(V_1,\cdots,V_N)= 
\prod_{i=1}^N \Bigl( \int_{\mathbb{R}^n} e^{-\alpha_i V_i(x_i)}  m_i(dx_i) \Bigr)^{\frac{1}{\alpha_i}}
$$
 with 
 $
\alpha_i = \frac{n+r_i}{np_i}
 $ on the set of even admissible tuples $\{V_i\}$.
Let  $\Phi_1, \cdots,\Phi_N$ be a maximizing tuple of (\ref{BSF0}).
 Then every function  $\Phi_j-\Phi_j(0)$ is $\beta_j$-homogeneous with 
$
\beta_j = \bigl(\sum_{i=1}^N \frac{1}{\alpha_i}\bigr)\bigl(1+ \frac{r_j}{n}\bigr).
$
\end{theorem}

Let us briefly describe the main results of the paper.
In section 2, we discuss the relation of the Blaschke--Santal{\'o} inequality to the optimal transportation problem and prove a multimarginal (in general non-Gaussian) version of the inequality 
    (\ref{BStransport}):
    $$
    \inf_{\pi \in \Pi(\nu_1,\cdots,\nu_N)} \int d(x) d\pi \le \sum_{i=1}^N \frac{1}{\alpha_i} {\rm Ent}_{\mu_i}(\nu_i).
$$
Here $\Pi(\nu_1,\cdots,\nu_N)$ is the set of measures with projections $\nu_1,\cdots,\nu_N$, 
$
d(x) = \sum_{i=1}^N \Phi_i(x_i)- c(x),
$
$\{\Phi_i\}$ is a maximizing tuple for 
    $\mathcal{BS}_{\alpha,m}$, and $\mu_i   = \frac{e^{-\alpha_i \Phi_i} }{\int e^{-\alpha_i \Phi_i} dm_i} \cdot m_i$.
    Theorem \ref{homogtheorem0}
is proved in section 3. We show in sections 4-5 
that $(p_1,\cdots,p_N)$-homogeneous cost functions admit other nice properties. In particular, we prove in section 4 that under
some additional assumptions on $c$,
the maximization problem for $\mathcal{BS}_{\alpha,m}$ is equivalent to the maximization problem for the Blaschke--Santal{\'o} set functional
$$
\prod_{i=1}^N m_i(K_i)^{\frac{1}{\alpha_1}}, 
$$
where  $\{K_i\}$ are symmetric, convex, and  $c(x) \le 1$ on $\prod_{i=1}^n K_i$. In  section 5, we discuss  the monotonicity of the   Blaschke--Santal\'o functional  under an appropriate version of the Steiner symmetrization procedure. In section 6, we consider examples. In particular, we prove a more general version of Theorem \ref{example1} of Kalantzopoulos and Saroglou.

\subsection*{Acknowledgements}
The  author was supported by RSF project 25-11-00007. 
The author thanks Matthieu Fradelizi for his interest and comments about the bibliography,
and Christos Saroglou for comments about Theorem \ref{example1}.

\section{Optimal transportation approach}

Let us remind the basic facts about the mass transportation problem. Given $N$ probability measures $\mu_i$ on $\mathbb{R}^n$ and a cost function $c(x) = c(x_1,\cdots,x_n)$, we consider the Kantorovich functional
$$
K_c(\pi) =  \int_{(\mathbb{R}^n)^N} c d\pi,
$$
where $\pi$ is a measure on $(\mathbb{R}^n)^N$ with marginal $\mu_i$: ${\rm Pr}_i \pi = \mu_i$. Here we are interested in the maximization problem. If the value of $K_c$ is maximized on some measure $\gamma$, we say that $\gamma$ is a solution to the Kantorovich maximization problem for the cost function $c$. For the existence of $\gamma$, it is sufficient to assume that $c$ is upper-semicontinuous and 
admits some integrability properties (see \cite{BKS},  { \cite{VillaniOldNew}}).

The corresponding linear dual functional $D$ is defined on tuples of functions $$(f_1(x_1),\cdots,f_N(x_N)),$$
satisfying the assumption
$\sum_{i=1}^N f_i(x_i) \ge c$,
and takes the form
$$
D(f_1,\cdots,f_N) = \sum_{i=1}^N \int_{\mathbb{R}^n} f_id\mu_i.
$$
We always have
$$
K_c \le D.
$$
Under broad assumptions on $c$, the dual functional $D$ admits a minimum point, and its minimum coincides with the maximum of the Kantorovich functional (no duality gap).
If the tuple $$(f_1(x_1),\cdots,f_N(x_N))$$ is a minimum point of $D$, the complementary slackness condition implies
$$
\sum_{i=1} f(x_i)=c(x)
$$
for $\gamma$-almost all $x =(x_1, \cdots,x_N)$.

Without loss of generality, we can always assume that 
our admissible tuple $\{f_i\}$ satisfies the following property:
$$
f_i(x_i) = \sup_{x_j, j \ne i} \bigl( c(x) - \sum_{j \ne i} f(x_j)\bigr),
$$
i.e., every $f_i$ is the multiple $c$-Legendre transform of the other functions (see \cite{Pass}). 

In what follows, we consider 
$N$-tuples of functions $V_1(x_1), \cdots,V_N(x_N)$, $x_i \in \mathbb{R}^n$, satisfying the assumption
$$
\sum_{i=1}^N V_i(x_i) \ge c
$$
for a given cost function $c=c(x_1,\cdots,x_N)$ (admissible tuples), and  the following Blaschke--Santal\'o-type functional:
\begin{equation}\label{BSF}
\mathcal{BS}_{\alpha,m}(V_1,\cdots,V_N)= 
\prod_{i=1}^N \Bigl( \int_{\mathbb{R}^n} e^{-\alpha_i V_i(x_i)}  m_i(dx_i) \Bigr)^{\frac{1}{\alpha_i}}.
\end{equation}
Here $m_i$ are some probability reference measures on $\mathbb{R}^n$, and  $\alpha_i$
are given positive numbers.

\begin{remark}
    In this section, we {\bf don't assume} that $m_i$ and $c$ are homogeneous.
\end{remark}

We associate with the tuple $\{V_i(x_i)\}$ a tuple of probability measures
\begin{equation}
\nu_i = \frac{e^{-\alpha_i V_i}}{\int e^{-\alpha_i V_i}dm_i} \cdot m_i, \ 1 \le i \le N.
\end{equation}

The Blaschke--Santal\'o functional
admits the following monotonicity property related to the
multimarginal optimal transportation problem with marginals $\nu_i$. In a less general form, it was first observed in \cite{KW}.

\begin{theorem} \label{monotone} (Monotonicity of the Blaschke--Santal\'o functional under optimal transportation). 
Assume that 
\begin{itemize}
    \item the optimal transportation problem 
    with marginals $\nu_i$ and the cost function $c$ admits a solution $\gamma \in \mathcal{P}((\mathbb{R}^{n})^N)$
    \item the corresponding dual transportation problem admits a solution 
    $\Phi_1, \cdots,\Phi_N$
    \item no duality gap holds.    
\end{itemize}
Then
$$
\prod_{i=1}^N \Bigl( \int_{\mathbb{R}^n} e^{-\alpha_i V_i}  dm_i \Bigr)^{\frac{1}{\alpha_i}}
\le \prod_{i=1}^N \Bigl( \int_{\mathbb{R}^n} e^{-\alpha_i \Phi_i}  dm_i \Bigr)^{\frac{1}{\alpha_i}}.
$$
\end{theorem}
\begin{proof}
By the assumption on $\{V_i\}$ and the complementary slackness condition
$$
\sum_{i=1}^N \Phi_i(x_i) = c(x) \le \sum_{i=1}^N V_i(x_i)
$$
  for $\gamma$-almost all $(x_1,\cdots,x_N)$. Take
   $$
\alpha = \frac{1}{\sum_{i=1}^N \frac{1}{\alpha_i}}
   $$
   and integrate the inequality
   $
1 \le e^{\alpha \sum_{i=1}^N{(V_i - \Phi_i)}}
   $
   over $\gamma$. By the H\"older inequality
   $$
1 \le \int_{(\mathbb{R}^n)^N} e^{\alpha \sum_{i=1}^N{(V_i - \Phi_i)}} d\gamma \le \prod_{i=1}^N \Bigl(\int_{\mathbb{R}^n} e^{\alpha_i(V_i-\Phi_i) }d\gamma\Bigr)^{\frac{\alpha}{\alpha_i}}
= \Bigl(\prod_{i=1}^N \frac{\int_{\mathbb{R}^n}  e^{-\alpha_i\Phi_i }dm_i}{\int_{\mathbb{R}^n}  e^{-\alpha_iV_i }dm_i}\Bigr)^{\frac{\alpha}{\alpha_i}}.
   $$
   The result follows immediately from this inequality.
\end{proof}

\begin{corollary} \label{maxim}
    Assume that the maximum of $\mathcal{BS}_{\alpha,m}(V_1,\cdots,V_N)$
over all admissible tuples is reached on the tuple $\{\Phi_i\}$. Assume, in addition, that all the assumptions of the previous theorem are fulfilled for the cost  $c$ and probability measures  $\mu_i = \frac{e^{-\alpha_i \Phi_i} }{\int e^{-\alpha_i \Phi_i} dm_i} \cdot m_i$. Then $\{\Phi_i\}$ solves the dual transportation problem for $c$ and $\mu_i$.
\end{corollary}

The following result is a generalization of the sharp Talagrand inequality (\ref{BStransport}) (see 
 \cite{Fathi}, \cite{Gozlan}, \cite{FGSZ}, \cite{CFM}). For $N>2$, it was known in a less general situation (see \cite{KW}).
Introduce the following new cost function:
$$
d(x) = \sum_{i=1}^N \Phi_i(x_i)- c(x) \ge 0.
$$

Given $N$ probability measures $\{\nu_i\}, 1 \le i \le N$, consider the multimarginal {\bf minimization} transportation problem for $d$:
$$
K^{\min}_d(\nu_1, \cdots,\nu_N) = \inf_{\pi \in \Pi(\nu_1,\cdots,\nu_N)} \int d(x) d\pi,
$$
where $\Pi(\nu_1,\cdots,\nu_N)$ is the set of measures with projections $\nu_1,\cdots,\nu_N$.

\begin{theorem}
Assume that the maximum of $\mathcal{BS}_{\alpha,m}(V_1,\cdots,V_N)$
over all admissible tuples is reached on the tuple $\{\Phi_i\}$.
    Assume that for a tuple of probability measures $\{\nu_i\}$ and the cost function $
d(x) = \sum_{i=1}^N \Phi_i(x)- c(x)
$, both the primal and dual transportation problems admit solutions and no duality gap  holds. Then 
    $$
K^{\min}_d(\nu_1,\cdots,\nu_N) \le \sum_{i=1}^N \frac{1}{\alpha_i} {\rm Ent}_{\mu_i}(\nu_i).
    $$
\end{theorem}
\begin{proof}
    Let $\{f_i(x_i)\}$ be a solution to the dual transportation problem for $\{\nu_i\}$ and $d$. In particular,
    $$
\sum_{i=1}^N f_i(x_i) \le \sum_{i=1}^N \Phi_i(x_i) - c(x).
    $$ 
    Since 
    $
\sum_{i=1}^N (\Phi_i - f_i)(x_i) \ge c(x), 
    $
    one has
    $
\prod_{i=1}^N \bigl(\int e^{\alpha_i(f_i - \Phi_i)} dm_i\bigr)^{\frac{1}{\alpha_i}} \le \bigl( \prod_{i=1}^N \int e^{-\alpha_i \Phi_i} dm_i\bigr)^{\frac{1}{\alpha_i}}.
    $
    by the optimality of $\{\Phi_i\}$.
    This inequality can be rewritten as follows:
    $$
\sum_{i=1}^N \frac{1}{\alpha_i} \log \int e^{\alpha_i f_i} d\mu_i \le 0.
    $$
    Let $\nu_i$ have densities with respect to $\mu_i$: $\nu_i = \rho_i \cdot\mu_i$ (otherwise, the result is trivial).
Applying the inequality $xy \le e^x-1 + y \log y$, one gets the result:
\begin{align*}K^{\min}_d(\nu_1,\cdots,\nu_N) & = \sum_{i=1}^N \int f_i d\nu_i \le 
\sum_{i=1}^N \frac{1}{\alpha_i} \int \bigl( \alpha_if_i -   \log \int e^{\alpha_i f_i} d\mu_i \bigr) \rho_i d\mu_i \\& \le
\sum_{i=1}^N  \frac{1}{\alpha_i} \bigl(\rho_i \log \rho_i -1 + e^{ \alpha_if_i -   \log \int e^{\alpha_i f_i} d\mu_i }  \bigr)d\mu_i
\\& = \sum_{i=1}^N \frac{1}{\alpha_i} \int \rho_i \log \rho_i d\mu_i = \sum_{i=1}^N \frac{1}{\alpha_i} {\rm Ent}_{\mu_i}(\nu_i).
\end{align*}

\end{proof}

Moreover, the following reverse Theorem is true.

\begin{theorem}
    Let $c$ be a cost function and $\{\Phi_i(x_i)\}$ be a tuple of functions satisfying
    $\sum_{i=1}^N \Phi(x_i) \ge c(x)$. Assume that the Kantorovich functional $K^{\min}_d$ with
    $d = \sum_{i=1}^N \Phi_i(x_i) - c(x)$
    satisfies inequality
    $$
K^{\min}_d(\nu_1, \cdots,\nu_N) \le \sum_{i=1}^N \frac{1}{\alpha_i} {\rm Ent}_{\mu_i}(\nu_i)
    $$
    for $\mu_i = \frac{e^{-\alpha_i \Phi_i(x_i)} dx_i}{\int e^{-\alpha_i\Phi_i} dx_i}$ and an arbitrary tuple of probability measures $\{\nu_i\}$. Then $\{\Phi_i\}$ is a maximum point of
    $\mathcal{BS}_{\alpha,m}$.
    \end{theorem}
\begin{proof}
    Let $\{V_i\}$ be a tuple of functions satisfying 
$\sum_i V_i \ge c.
$ Define probability measures $
\nu_i = \frac{e^{-\alpha_i V_i} dx_i}{\int e^{-\alpha_i V_i}dx_i}.
$
One has
$
{\rm Ent}_{\mu_i}(\nu_i)
= \int (\alpha_i \Phi_i -\alpha_i V_i )d\nu_i
+ \log \bigl( \frac{\int e^{-\alpha_i \Phi_i}dx_i}{\int e^{-\alpha_i V_i}dx_i}\bigr).
$
Hence
$$
\sum_{i=1}^N \frac{1}{\alpha_i} {\rm Ent}_{\mu_i}(\nu_i) = \sum_{i=1}^N (\Phi_i -V_i)d\nu_i
+ \log \prod_{i=1}^N \Bigl(\frac{\int e^{-\alpha_i \Phi_i}dx_i}{\int e^{-\alpha_i V_i}dx_i}\Bigr)^{\frac{1}{\alpha_i}}.
$$
On the other hand
$\sum_{i=1}^N (\Phi_i - V_i) \le  \sum_{i=1}^N \Phi_i -c$. Hence, by duality (the primal functional dominates the  dual functional in the minimization problem)
$$
K^{\min}_d(\nu_1, \cdots,\nu_N)
\ge \sum_{i=1}^N  \int (\Phi_i - V_i) d\nu_i.
$$
Finally, the assumption of the theorem implies:
$$
\sum_{i=1}^N  \int (\Phi_i - V_i) d\nu_i \le K^{\min}_d(\nu_1, \cdots,\nu_N)
\le  \sum_{i=1}^N \frac{1}{\alpha_i} {\rm Ent}_{\mu_i}(\nu_i) = \sum_{i=1}^N (\Phi_i -V_i)d\nu_i
+ \log \prod_{i=1}^N \Bigl(\frac{\int e^{-\alpha_i \Phi_i}dx_i}{\int e^{-\alpha_i V_i}dx_i}\Bigr)^{\frac{1}{\alpha_i}}.
$$
Thus, one gets the desired inequality.
\end{proof}

\section{Homogeneity of the maximizers}

In this section, we generalize Theorem B from \cite{CKLR} and find sufficient conditions for the homogeneity of the maximizers of $\mathcal{BS}_{\alpha,m}$. 
It will always be assumed that the measures $m_i$ satisfy the assumption (\ref{dens-hom}):
$$
\rho_i(tx) = t^{r_i} \rho(x)
$$
for all $t \ge 0$ and $x_i \in \mathbb{R}^n$.
The cost function $c$ is assumed to be at least $p$-homogeneous:
$$
c(tx) = t^p c(x)
$$
for all $t \ge 0$ and $x \in (\mathbb{R}^n)^N$, but for the main result, we ask for the stronger assumption of $(p_1,\cdots,p_N)$-homogeneity:
$$
c(x_1, \cdots, t_i x_i, \cdots,x_N)
= t^{p_i}_i c(x)
$$
for all $1 \le i \le N, x \in (\mathbb{R}^n)^N, t_i \ge 0$.

Recall the standard notation
$$
{\rm Var}_{\mu}f = \int f^2 d\mu - \Bigl(\int f d\mu \Bigr)^2.
$$

\begin{lemma}
\label{1and2order}
Consider $\mathcal{BS}_{\alpha,m}$ functional, where $c$ is $p$-homogeneous and every $m_i$ satisfies  (\ref{dens-hom}). 
Let $\{\Phi_i\}$ be a tuple of maximizers for $\mathcal{BS}_{\alpha,m}$ and $\mu_i = \frac{ e^{-\alpha_i \Phi_i}}{\int e^{-\alpha_i \Phi_i}dx_i}dx_i$.
Then the following is true:
\begin{equation}
    \label{firstorder}
    \int c d \gamma =\sum_{i=1}^{N} \int \Phi_i d\mu_i =  \sum_{i=1}^N \frac{n+r_i}{p\alpha_i}
\end{equation}
\begin{equation}
    \label{secondorder}
    \sum_{i=1}^N \alpha_i {\rm Var}_{\mu_i}(\Phi_i) \le    \sum_{i=1}^N \frac{n+r_i}{p\alpha_i}
\end{equation}
\begin{equation}
    {\rm Var}_{\gamma}(c) \le \bigl({\sum_{i=1}^N \frac{1}{\alpha_i}}\bigr)\sum_{i=1}^N \alpha_i {\rm Var}_{\mu_i} (\Phi_i).
\end{equation}
In addition, one has equality in the last inequality if and only if $$\alpha_i \bigl(\Phi_i(x_i) -\int \Phi_i d\mu_i\bigr) = \alpha_j \bigl(\Phi_j(x_j) -\int \Phi_j d\mu_j\bigr)$$ $\gamma$-almost everywhere.
\end{lemma}
\begin{proof}
    We observe that the functions $\{\Phi_{i,t}(x_i) = \frac{\Phi_i(tx_i)}{t^p}\}$ satisfy the same constraint
    $\sum_{i=1}^N \Phi_{i,t}(x_i) \ge c(x)$ for all $t>0$. Hence
    $f(t) = \mathcal{BS}_{\alpha,m}(\Phi_{1,t},\cdots,\Phi_{N,t}) \le \mathcal{BS}_{\alpha,m}(\Phi_{1},\cdots,\Phi_{N}) $. 
Applying the change of variables $tx_i=y_i$ and using the homogeneity of $m_i$, one gets
    $$
\mathcal{BS}_{\alpha,m}(\Phi_{1,t},\cdots,\Phi_{N,t})
= \prod_{i=1}^N \Bigl( \int e^{-\frac{\alpha_i}{t^p} \Phi_i(tx_i)} \rho_i(x_i)dx_i\Bigr)^{\frac{1}{\alpha_i}}
= \frac{1}{t^{\sum_{i=1}^N \frac{n+r_i}{\alpha_i}}}\prod_{i=1}^N \Bigl( \int e^{-\frac{\alpha_i}{t^p} \Phi_i(y_i)} \rho_i(y_i)dy_i\Bigr)^{\frac{1}{\alpha_i}}.
    $$
    Let $s = \frac{1}{t^p}$. The function 
    $$
g(s) = \log \mathcal{BS}_{\alpha,m}(\Phi_{1,s^{-\frac{1}{p}}},\cdots,\Phi_{N,s^{-\frac{1}{p}}}) = \sum_{i=1}^N \frac{1}{\alpha_i} \log \int e^{-\alpha_i s \Phi_i(y)} \rho_i(y)dy + \log s \sum_{i=1}^N \frac{n+r_i}{p\alpha_i}
    $$
    attains its maximum at $s=1$. Hence
    $$
g'(s) = \frac{1}{s}\sum_{i=1}^N \frac{n+r_i}{p\alpha_i} - 
\sum_{i=1}^N\frac{1}{Z_i}  \int \Phi_i e^{-\alpha_i s \Phi_i(y)} \rho_i(y)dy,
    $$
    $$
g''(s) = -\frac{1}{s^2}\sum_{i=1}^N \frac{n+r_i}{p\alpha_i} + 
\sum_{i=1}^N\alpha_i \bigl[\frac{1}{Z_i} {\int \Phi^2_i e^{-\alpha_i s \Phi_i(y)} \rho_i(y)dy}
- \Bigl( \frac{1}{Z_i} { \int \Phi_i e^{-\alpha_i s \Phi_i(y)} \rho_i(y)dy}\Bigr)^2\bigr],
    $$
    where $Z_i =\int e^{-\alpha_i s \Phi_i(y)} \rho_i(y)dy $.
    Finally, (\ref{firstorder}) follows from the first order condition $g'(1)=0$ and (\ref{secondorder}) follows from the second order condition $g''(1) \le 0$.

    Let us show the third inequality. Define for brevity
$\hat{\Phi}_i = 
\Phi_i - \int \Phi_i d\mu_i$. We apply the following identity:
\begin{equation}
    \label{quadratic}
\sum_{i=1}^N \alpha_i (\hat{\Phi}_i)^2
= \frac{1}{\sum_{i=1}^N \frac{1}{\alpha_i}}\bigl[  \bigl(\sum_{i=1}^N \hat{\Phi}_i\bigr)^2 + \frac{1}{\alpha_i\alpha_j}\sum_{i\ne j} (\alpha_i \hat{\Phi}_i - \alpha_j \hat{\Phi}_j )^2 \bigr].
\end{equation}
One has
$$
\sum_{i=1}^N \alpha_i \int(\hat{\Phi}_i)^2 d\mu_i = 
\int \sum_{i=1}^N \alpha_i (\hat{\Phi}_i)^2 d\gamma \ge 
\frac{1}{\sum_{i=1}^N \frac{1}{\alpha_i}}\int \bigl(\sum_{i=1}^N \tilde{\Phi}_i\bigr)^2 d\gamma= 
\frac{1}{\sum_{i=1}^N \frac{1}{\alpha_i}}\bigl[\int c^2  d\gamma
 - \bigl( \int c d\gamma \bigr)^2 \bigr].$$
 It is clear, that for the equality it is necessary and sufficient to have $\alpha_i \hat{\Phi}_i = \alpha_j \hat{\Phi}_j $ $\gamma$-a.e.
\end{proof}

Finally, the following theorem provides sufficient conditions for the homogeneity of an arbitrary maximizing tuple.

\begin{theorem}
\label{homogtheorem}
Let the measures $m_i$ satisfy the assumption (\ref{dens-hom}), and $c$ be $(p_1,\cdots,p_N)$-homogeneous. 
Consider the functional (\ref{BSF})
 with 
 $$
\alpha_i = \frac{n+r_i}{np_i}.
 $$
Let  $\{\Phi_i\}$ be a maximizing tuple of (\ref{BSF}).
Assume, in addition, that
\begin{itemize}
    \item the cost function $c$ is locally Lipschitz,
    \item every function $\Phi_j$ is convex and $\Phi_j(0)<\infty$.
\end{itemize}
 Then every function  $\Phi_j-\Phi_j(0)$ is $\beta_j$-homogeneous with 
$$
\beta_j = \bigl(\sum_{i=1}^N \frac{1}{\alpha_i}\bigr)\bigl(1+ \frac{r_j}{n}\bigr).
$$
\end{theorem}

\begin{remark}
     Here we don't recover the result from \cite{NT} on the homogeneity of the maximizers for $c_{bar}$, because the barycentric cost function is  not $(p_1, \cdots,p_N)$-homogeneous. 
\end{remark}

\begin{proof}
By Corollary \ref{maxim} the tuple $\{\Phi_i\}$ is a solution to the dual multimarginal transportation problem for $c$ and 
$$
\mu_i = \frac{1}{Z_i}{e^{-\alpha_i \Phi_i}\rho_i dx_i},$$ 
where $Z_i = {\int e^{-\alpha_i \Phi_i} \rho_i dx_i}
$. In particular,
$\sum_{i=1}^N \Phi_i(x_i) =c(x)$
for $\gamma$-almost all $x=(x_1,\cdots,x_N)$, where $\gamma$ is a solution to the primal transportation problem.
In addition, since $
\sum_{i=1}^N \Phi_i(x_i) \ge c(x)
$ for all $x$, it follows from the regularity assumptions on $c$ and $\Phi_i$ that 
$$
\nabla \Phi_i(x_i) = \nabla_{x_i} c(x)
$$
$\gamma$-a.e. Taking the scalar product of the latter identity, we get
\begin{equation}
\label{cost(xi)}
\langle  \nabla \Phi_i(x_i),x_i \rangle  = \langle  \nabla_{x_i} c(x),x_i \rangle = p_i c(x). 
\end{equation}

According to Lemma \ref{1and2order}
$
 \int c d\gamma =\frac{1}{\sum_{i=1}p_i} \sum_{i=1}^N \frac{n+r_i}{\alpha_i}=n.
$
Applying (\ref{cost(xi)}), we get, in particular,
$$
n = \int c d\gamma = \frac{1}{p_i} \int \langle x_i, \nabla \Phi_i(x_i)\rangle d \mu_i = \frac{1}{p_i Z_i}
\int_{\mathbb{R}^n} \langle x_i, \nabla \Phi_i(x_i)\rangle e^{-\alpha_i \Phi_i} \rho_i dx_i.
$$
Without loss of generality, we assume that $\Phi_i$ is lower-semicontinuous and apply the following integration by parts formula 
$$
\int_{\mathbb{R}^n} {\rm{div}}(v) e^{-\alpha_i \Phi_i} dx_i = \alpha_i \int_{\Omega} \langle v, \nabla \Phi_i \rangle e^{-\alpha_i \Phi_i} dx_i + \int_{\partial \Omega} \langle v, \nu\rangle e^{-\alpha_i \Phi_i} d \mathcal{H}^{n-1}.
$$
Here $\Omega = \{ \Phi_i < \infty\}$, $\nu$ is the outer normal vector on $\partial {\Omega}$, $\mathcal{H}^{n-1}$ is the $n-1$-dimensional Haussdorff measure on $\partial \Omega$, and $v$ is a sufficiently regular vector field.

In particular, we obtain
$$
n = \frac{1}{p_i\alpha_i Z_i}\bigl[  \int {\rm{div}}\bigl(\rho_i\cdot x_i \bigr) e^{-\alpha_i \Phi_i} dx_i  - \int_{\partial \Omega} \langle x_i, \nu \rangle e^{-\alpha_i \Phi_i} \rho_i d\mathcal{H}^{n-1} \bigr].
$$
Using that ${\rm{div}}\bigl(\rho_i\cdot x_i \bigr)  = (n+r_i) \rho_i$, one gets
$
n = \frac{n+r_i}{p_i\alpha_i} - \frac{1}{p_i\alpha_i Z_i}
{ \int_{\partial \Omega} \langle x_i, \nu \rangle e^{-\alpha_i \Phi_i} \rho_i d\mathcal{H}^{n-1}}.
$
Since $\frac{n+r_i}{p_i\alpha_i}  = n$, one finally gets
\begin{equation}
    \label{vanishingdpo}
\int_{\partial \Omega} \langle x_i, \nu \rangle e^{-\alpha_i \Phi_i} \rho_i d\mathcal{H}^{n-1} =0.
\end{equation}
Next, we observe that $\langle x_i, \nu \rangle \ge 0$ because  $\Omega$ is convex as a sublevel set of the convex function $\Phi_i$ and contains the origin.
This implies that either $\partial \Omega$ is empty or $\Phi_i|_{\partial \Omega}=+\infty$ $\mathcal{H}^{n-1}$-a.e. (equivalently, $\lim_{x \to y, x\in {\rm Int}(\Omega)} \Phi_i(x)=+\infty$ for 
$\mathcal{H}^{n-1}$-almost all $y \in \partial \Omega$).

Finally, one gets that the following integration by parts formula is valid:
 \begin{equation}
 \label{ibppdo}
\int_{\mathbb{R}^n} {\rm div}(v) e^{-\alpha_i \Phi_i}dx_i
= \alpha_i\int_{\mathbb{R}^n} \langle v, \nabla \Phi_i\rangle e^{-\alpha_i \Phi_i}dx_i.
    \end{equation}
Clearly, the formula works for vector fields that are sufficiently regular and integrable.

Next, we apply another result of Lemma \ref{1and2order}:
\begin{equation}
    \label{0111}
\sum_{i=1}^N \alpha_i {\rm Var}_{\mu_i} (\Phi_i) \le \frac{1}{\sum_{i=1}p_i} \sum_{i=1}^N \frac{n+r_i}{\alpha_i}=n.
\end{equation}

 Using that $c = \sum_{i=1}^N \Phi_i = \langle x_i,\nabla \Phi_i(x_i)\rangle$
 $\gamma$-a.e. and the integration by parts formula (\ref{ibppdo}), we compute $\int c^2 d\gamma$:
\begin{align*}
 \int c^2 d\gamma & = \sum_{i=1}^n \int c \Phi_i d\gamma = 
  \sum_{i=1}^n \frac{1}{p_i}\int  \langle x_i \nabla \Phi_i \rangle \Phi_i d\mu_i = \sum_{i=1}^N\frac{1}{p_i\alpha_i Z_i} {\int {\rm div}(x_i \cdot \Phi_i \rho_i) e^{-\alpha_i \Phi_i} dx_i}
  \\& = \sum_{i=1}^N \bigl[ \frac{n+r_i}{p_i\alpha_i} \int \Phi_i d\mu_i + \frac{1}{p_i\alpha_i} \int \langle x_i,\nabla \Phi_i \rangle d\mu_i\bigr]
  = n \sum_{i=1}^N \int \Phi_i d\mu_i
  + \Bigl(\sum_{i=1}^n \frac{1}{\alpha_i}\Bigr) \int cd\gamma \\& = n^2 + \Bigl(\sum_{i=1}^n \frac{1}{\alpha_i}\Bigr) n.
\end{align*}

Thus, we get
$
\int c^2  d\gamma
 - \bigl( \int c d\gamma \bigr)^2 = \bigl(\sum_{i=1}^n \frac{1}{\alpha_i}\bigr)n.
$
Hence, one gets from (\ref{quadratic}) that
$\sum_{i=1}^N \alpha_i {\rm Var}_{\mu_i}(\Phi_i) \ge n$, Finally, we get by the opposite inequality (\ref{0111}) that 
$\sum_{i=1}^N \alpha_i {\rm Var}_{\mu_i}(\Phi_i) = n$. This can be  possible only if
$\alpha_i (\Phi_i(x_i) -\int \Phi_i d\mu_i) = \alpha_j (\Phi_j(x_j) -\int \Phi_j d\mu_j)$ $\gamma$-a.e. Since
$$
c(x) = \sum_{i=1}^N \Phi_i(x_i) = \frac{1}{p_j} \langle x_j, \nabla \Phi_j(x_j)\rangle, 
$$
we obtain that for $\mu_j$-almost all $x_j$ one has
$$
\frac{1}{p_j} \langle x_j, \nabla \Phi_j(x_j)\rangle = \Phi_j(x_j) \bigl( 1  + \sum_{i \ne j} \frac{\alpha_j}{\alpha_i}\bigr) + C_j
$$
for some constant $C_j$. 
Taking $x_j=0$, we get that $C_j = - \Phi_j(0) \bigl( 1  + \sum_{i \ne j} \frac{\alpha_j}{\alpha_i}\bigr)$. Hence
$$
\Phi_j(x_j) - \Phi_j(0)= \frac{1}{\sum_{i=1}^N \frac{1}{\alpha_i}}\frac{1}{\alpha_j p_j} \langle x_j, \nabla \Phi_j(x_j)\rangle =  \frac{1}{\sum_{i=1}^N \frac{1}{\alpha_i}}\frac{n}{n+r_j} \langle x_j, \nabla \Phi_j(x_j)\rangle.
$$
This completes the proof.
\end{proof}

\begin{corollary}
\label{maincorollary}
    Let  the measures $m_i$ satisfy the assumption (\ref{dens-hom}) and the  densities $\rho_i$ are positive almost everywhere.
Consider the functional  (\ref{BSF}) on the set of even functions
and   let $\{\Phi_i\}$ be a maximizing tuple. Assume that
\begin{enumerate} 
\item 
$c$ is a continuous $(p_1,\cdots,p_N)$-homogeneous function for some numbers $p_i>0$
\item for every 
$1 \le i \le N$, there exist numbers 
$s_j \in \{-1,1\}$; and $j \ne i$, such that 
$$
c(x_1, \cdots, x_{i-1}, - x_i, x_{i+1},\cdots,x_N) = 
c(s_1x_1, \cdots,  s_{i-1}x_{i-1}, x_i,  s_{i+1}x_{i+1},\cdots, s_N x_N) 
$$
\item 
  $c$ is  convex with respect to every variable, i.e., for all $i \in [1,\cdots,N]$ and $x_j \in \mathbb{R}^n$, $j \ne i$, the mapping
  $$
x \to c(x_1, \cdots, x_{i-1},x, x_{i+1}, x_N)
$$ is a convex function.
\end{enumerate}
Then the conclusion of Theorem \ref{homogtheorem} holds.
\end{corollary}
\begin{proof}
We have to prove that every $\Phi_i$ is convex (together with the symmetry assumption, this clearly implies that $\Phi_i(0)=\min \Phi_i < +\infty$). 
    Every  $\Phi_i$ must satisfy
    $$
\Phi_i(x_i) =\tilde{\Phi}_i(x_i) =\sup_{x_j, j \ne i} \bigl( c(x) - \sum_{j\ne i} \Phi_j(x_j)\bigr),
    $$
otherwise, replacing $\Phi_i$ with $\tilde{\Phi}_i$ will increase the value of the functional.
We observe that $\tilde{\Phi}_i$ is even, provided all $\Phi_j$ with $j \ne i$, are even.
Indeed, this follows immediately from the assumption (2) and the evenness of $\Phi_j$, $j \ne i$.
    Clearly,     $\Phi_i$ is convex as a supremum of convex functions. 
\end{proof}

\begin{remark}
    Corollary \ref{maincorollary} implies, in particular, Theorem B from \cite{CKLR} on the  homogeneity of the maximizers of the functional
  $$
\mathcal{F}_{p,\Psi}(V) = \Bigl( \int_{\mathbb{R}^n} e^{-V} dx\Bigr)^{\frac{1}{p}} 
\Bigl( \int_{\mathbb{R}^n} e^{-\frac{1}
{p-1}V^*} \det D^2 \Psi^*dy\Bigr)^{\frac{1}{q}}.
$$
Here $\Psi$ is a given convex $p$-homogeneous function. The functional is considered on the set of even convex functions.  The proof from \cite{CKLR} relies on the application of the maximum principle for the corresponding Monge--Amp\`ere equation.
\end{remark}

\begin{remark}
    The convexity assumptions of Corollary \ref{maincorollary}
 and Theorem \ref{homogtheorem} seem to be too strong. Eventually, the  conclusion of  Theorem \ref{homogtheorem} holds if the functions $\Phi_i$ have the following property: every function $f_i(t) = \Phi_i(tx)$
 is increasing on $[0,+\infty)$ for all $x$. Equivalently, the sublevel sets of $\Phi_i$ are starshaped (star sets). We don't prove any results of this type here.
\end{remark}

In the rest of the section, we discuss
how to provide an equivalent formulation of our problem on $\mathbb{S}^{n-1}$. Under the assumptions of Corollary \ref{maincorollary}
 (or Theorem \ref{homogtheorem}),
the maximization problem for  $\mathcal{BS}_{\alpha,m}$ can have only a homogeneous solution $\{\Phi_i\}$ (up to the addition of constants). 
Without loss of generality, we assume  that our maximizing tuple satisfies
$$
\Phi_i(0)=0
$$
for all $1 \le i \le N$. 
Let us represent every $\Phi_i$ in the form
$$
\Phi_i(t \theta_i) = \frac{p_i}{\beta_i} t^{\beta_i}\varphi^{\beta_i}_i(\theta_i
$$
for some positive function $\varphi_i$ on $\mathbb{S}^{n-1}$.

\begin{lemma}\label{homogadmissible}
    Admissibility of the tuple $\{\Phi_i\}$ is equivalent to the property
$$
c(\theta_1, \cdots,\theta_N) \le \prod_{i=1}^N \varphi_i^{p_i}(\theta_i).
$$
\end{lemma}
\begin{proof}
    Assume that $c(\theta_1, \cdots,\theta_N) \le \prod_{i=1}^N \varphi_i^{p_i}(\theta_i)$, equivalently
    $
c\bigl(\frac{\theta_1}{\varphi_1(\theta_1)}, \cdots,\frac{\theta_N}{\varphi_N(\theta_N)}\bigr) \le 1.
    $
We get by the H{\"o}lder inequality
$$c(x) = \prod_{i=1}^N r_i^{p_i} \cdot c(\theta_1, \cdots,\theta_N) \le 
\prod_{i=1}^N \varphi_i^{p_i}(\theta_i) r_i^{p_i}  \le  \sum_{i=1}^N
 \frac{p_i}{\beta_i} \varphi_i^{\beta_i}(\theta_i) r_i^{\beta_i} = \sum_{i=1}^N \Phi_i(x_i).$$ Thus, the tuple $\{\Phi_i\}$ is admissible.
 
If $\{\Phi_i\}$ is admissible, then
$
c\bigl(\frac{\theta_1}{\varphi_1(\theta_1)}, \cdots,\frac{\theta_N}{\varphi_N(\theta_N)}\bigr) \le  \sum_{i=1}^N
 \frac{p_i}{\beta_i} \varphi_i^{\beta_i}(\theta_i) \bigl(  \frac{1}{\varphi_i(\theta_i)}\bigr)^{\beta_i}
 = \sum_{i=1}^N \frac{p_i}{\beta_i}=1.
$
The proof is complete.
\end{proof}
 
The functional 
 $\mathcal{BS}_{\alpha,m}$ restricted to homogeneous functions $\{V_i\}$ can be rewritten in terms of functions $\{\varphi_i\}$:
 \begin{align*}
    \mathcal{BS}_{\alpha,m}(V_1,\cdots,V_N) & = \prod_{i=1}^N \bigl( \int_{\mathbb{S}^{n-1}} e^{-{\alpha_i} V_i(x_i)} \rho_i(x_i)dx_i\bigr)^{\frac{1}{\alpha_i}}
    \\& =  \prod_{i=1}^N \Bigl( \int_{\mathbb{S}^{n-1}} \int_0^{\infty} e^{-\frac{\alpha_ip_i}{\beta_i} \varphi_i^{\beta_i}(\theta_i) t^{\beta_i}_i} \rho_i(\theta_i) t_i^{n-1+ r_i}dt_i d\theta_i\Bigr)^{\frac{1}{\alpha_i}}.
 \end{align*}
Making the change of variables 
$s_i = t_i \varphi_i(\theta_i)$ we conclude that 
$$
\mathcal{BS}_{\alpha,m}(V_1,\cdots,V_N) = C  \prod_{i=1}^N \Bigl( \int_{\mathbb{S}^{n-1}} \frac{\rho_i(\theta_i) 
d\theta_i}{\varphi^{n+r_i}_i(\theta_i)}\Bigr)^{\frac{1}{\alpha_i}},
$$
where $C$ is a constant, depending on parameters $\alpha_i,p_i,r_i,n$.

In conclusion, we obtain the following statement.

\begin{proposition} \label{1002} Under the above assumptions, the problem of maximizing $\mathcal{BS}_{\alpha,m}$ is equivalent to the maximization problem of the spherical Blaschke--Santal{\'o} type functional
\begin{equation}
\label{BSSn-1}
\prod_{i=1}^N \Bigl( \int_{\mathbb{S}^{n-1}} \frac{\rho_i(\theta_i)d\theta_i}{\varphi^{n+r_i}_i(\theta_i)}\Bigr)^{\frac{1}{\alpha_i}}
\end{equation}
under constraints
\begin{equation}
\label{product-constr}
c(\theta_1, \cdots,\theta_N) \le \prod_{i=1}^N \varphi_i^{p_i}(\theta_i).
\end{equation}
\end{proposition}
 Remark that the latter assumption can be rewritten in the additive (transportation-type) form:
$$
\log c(\theta_1, \cdots,\theta_N) \le \sum_{i=1}^N p_i \log \varphi_i(\theta_i),
$$
which makes this spherical problem more similar to the initial problem. See \cite{Kolesnikov2018}, \cite{FGSZ}, and the reference therein for the study of the logarithmic cost function 
$$
\log\langle x, y \rangle
$$
in the case $N=2$.

Finally, the following Theorem 
generalizes Theorem \ref{monotone} and Corollary \ref{maxim}.
\begin{theorem}
\label{BSsn-12}
    Let $c \ge 0$ be a sufficiently regular cost function on $(\mathbb{S}^{n-1})^N$, $\varphi_i$ be nonnegative functions satisfying (\ref{product-constr}) and, in addition, $\int_{\mathbb{S}^{n-1}} \frac{\rho_i(\theta_i)d\theta_i}{\varphi^{n+r_i}_i(\theta_i)} < \infty$ for all $1 \le i \le N$. Consider the mass transportation problem with the cost function
    $\log c$ and probability measures $$\mu_i = \frac{1}{Z_i} \frac{\rho_i(\theta_i)d\theta_i}{\varphi^{n+r_i}_i(\theta_i)},$$
    where $Z_i = {\int_{\mathbb{S}^{n-1}} \frac{\rho_i(\theta_i)d\theta_i}{\varphi^{n+r_i}_i(\theta_i)}} $ are the normalizing constants.
    Let $\{u_i\}, 1 \le i \le N$ be a solution to the corresponding dual problem.
    Define $\psi_i = e^{\frac{1}{p_i} u_i}$. If no duality gap for the cost $\log c$ holds, the following inequality is true
    $$
    \prod_{i=1}^N \Bigl( \int_{\mathbb{S}^{n-1}} \frac{\rho_i(\theta_i)d\theta_i}{\varphi^{n+r_i}_i(\theta_i)}\Bigr)^{\frac{1}{\alpha_i}}
\le 
\prod_{i=1}^N \Bigl( \int_{\mathbb{S}^{n-1}} \frac{\rho_i(\theta_i)d\theta_i}{\psi^{n+r_i}_i(\theta_i)}\Bigr)^{\frac{1}{\alpha_i}}.
    $$
    In particular, if $\{\varphi_i\}$ is a  maximum point for (\ref{BSSn-1}), then  $\varphi_i = c_i e^{\frac{1}{p_i}u_i}$, where the constants $c_i>0$ satisfy $\prod_{i=1}^N c_i=1$.
\end{theorem}
  \begin{proof}
      Let $\gamma$ be a solution to the primal transportation problem. One has
      $$
\log c = \sum_{i=1}^N u_i = \sum_{i=1}^N p_i \log \psi_i
      $$
      $\gamma$-a.e. Hence, for $\gamma$-almost all $(\theta_1, \cdots, \theta_N)$, one has
      $
\prod_{i=1}^N \psi_i^{p_i}(\theta_i)
\le \prod_{i=1}^N \varphi_i^{p_i}(\theta_i).
      $
      Hence, by the H{\"o}lder inequality
      \begin{align*}
          1 \le \int \prod_{i=1}^N \frac{\varphi_i^{p_i}(\theta_i)}{\psi_i^{p_i}(\theta_i)} d\gamma
          \le \prod_{i=1}^N \Bigl( \int \frac{\varphi_i^{n+r_i}(\theta_i)}{\psi_i^{n+ r_i}(\theta_i)} d\gamma \Bigr)^{\frac{p_i}{n+r_i}}
          =  \prod_{i=1}^N \Bigl(\frac{1}{Z_i} {\int_{\mathbb{S}^{n-1}} \frac{\rho_i(\theta_i)d\theta_i}{\psi^{n+r_i}_i(\theta_i)}}\Bigr)^{\frac{1}{n\alpha_i}}.
      \end{align*}
      This completes the proof.
  \end{proof}

\begin{remark}
    A typical example of $c$ is $\langle x, y \rangle$ and the cost $\log c$ is defined to be $+\infty$ outside of the area $\langle x, y \rangle >0$. Thus, the questions of existence/duality are not obvious. However, everything works in the symmetric setting. 
\end{remark}

\subsection{The classical Blaschke--Santal\'o inequality revisited}

Let us quickly discuss the transportation proofs of the   classical Blaschke--Santal\'o inequality. 
By the results of this section, any maximizer $\Phi$ 
of the functional 
$$
\int_{\mathbb{R}^n} e^{-V} dx 
\int_{\mathbb{R}^n} e^{-V^*} dy 
$$
(if it exists) has the form
$$
\Phi = \frac{1}{2} \varphi^2(\theta) |x|^2.
$$
This was proved in \cite{CKLR}, but Theorem 
\ref{homogtheorem} provides more soft arguments.
The existence of a maximizer can be proved by the compactness arguments (see \cite{CKLR}). By Corollary  \ref{maxim}, $\nabla \Phi$ pushes forward
$\mu = \frac{e^{-\Phi} dx}{\int e^{-\Phi}dx}$
onto $\nu = \frac{e^{-\Phi^*}dy}{\int e^{-\Phi^*}dy}$ and satisfies
\begin{equation}
\label{KE}
\frac{e^{-\Phi}}{\int e^{-\Phi}dx} = \frac{e^{-\Phi^*}(\nabla \Phi)}{\int e^{-\Phi^*}dy} \det D^2 \Phi.
\end{equation}
Since $\Phi^*(\nabla \Phi) = \Phi$, we also conclude that $\det{D^2 \Phi}$ is constant. Using Theorem \ref{BSsn-12}, one can associate with $\varphi$ a Monge--Amp\`ere equation on $\mathbb{S}^{n-1}$ (see \cite{BFR}, \cite{CKLR}). Finally, we can conclude that $\Phi$ is a quadratic function in several ways. Namely, one can use  1) Pogorelov's theorem describing global solutions to the equation $\det{D^2 \Phi}=1$, 2) the main result of \cite{CagKolWer} about solutions to (\ref{KE}), 3) the results on uniqueness (up to linear transformations) for the corresponding  Monge--Amp\`ere equation on the sphere.

\section{Inequalities for sets and for functions}

In this section, we continue our study of the homogeneous solutions and find a direct link between the functional and the set inequality. 
This was partially done in the previous sections, where we find conditions ensuring the homogeneity of the maximizers in the functional inequality.  Here we discuss how to recover the functional inequality from the set inequality using the "layer-cake" representation. 

{\bf Assumption on the reference measures:}
$$
m_i = \rho_i(x_i)dx_i,
$$
where $\rho_i$ is $r_i$-homogeneous.

{\bf Assumption on the cost $c$:}
\begin{itemize}
    \item $c$ is {\bf convex} for every $x_i$, i.e.,
$$
x \to c (x_1, \cdots, x_{i-1}, x, x_{i+1}, \cdots, x_N)
$$ is convex
\item $c$ is {\bf $p$-homogeneous} for some $p>0$:
$$
c(tx) = t^p c(x), \ t\ge 0.
$$
\item 
for every 
$1 \le i \le N$, there exist numbers 
$s_j \in \{-1,1\}$, $j \ne i$ such that 
$$
c(x_1, \cdots, x_{i-1}, - x_i, x_{i+1},\cdots,x_N) = 
c(s_1x_1, \cdots,  s_{i-1}x_{i-1}, x_i,  s_{i+1}x_{i+1},\cdots, s_N x_N) 
$$
\end{itemize}
We have already seen in the previous section that, under these assumptions, the multiple Legendre transform preserves convex and even functions. More precisely, every function
$$
f_i(x_i) = \sup_{x_j, j \ne i} \bigl( c(x) - \sum_{j \ne i} f_j(x_j)\bigr)
$$
is convex. Moreover, $f_i$ is even provided all $f_j$ are even, $j \ne i$.

We consider a Blaschke--Santal{\'o}-type functional  on the tuples of {\bf even} functions. 
 Since the Lagrange transform preserves evenness and convexity, without loss of generality, one can consider only convex even tuples $\{V_i\}$. 
 
 \begin{remark}
Note that $\sum_{i=1}^N V_i(0)\ge c(0)=0$. Since $0$ is the minimum point for every $V_i$, without loss of generality, we can assume that $V_i \ge 0$ and $V_i(0)=0$.
 \end{remark}

Our aim is to relate the $\mathcal{BS}_{\alpha,m}$ functional to the following 
 Blaschke--Santal{\'o} functional for sets:
$$
\mathcal{BS}^{set}_{m,\alpha}(K_1, \cdots,K_N) = \prod_{i=1}^N m_i(K_i)^{\frac{1}{\alpha_i}}.
$$

Similarly, we consider only {\bf symmetric and convex } sets.
The {\bf $c$-polar transform}
$$
K_i = \{x_i : c(x) \le 1 \ \ {\rm if}\ \  x_j \in K_j, j \ne i\}
$$
preserves symmetric convex  sets as well.

We are interested in the following set volume product maximization problem: 
find the maximum of $
\mathcal{BS}^{set}_{m,\alpha}
$ on the tuples of {even  convex} sets  $\{K_i\}$ satisfying
\begin{equation}
\label{setcond2}
c(x) \le 1
\end{equation}
for all $x \in \prod_{i=1}^N K_i$.

In the following Proposition and Corollary, we prove a generalization of the well-known result about the equivalence of the set and functional versions of the classical Blaschke--Santal{\'o} inequality (see \cite{Ball1986}, \cite{AAKM}). 

\begin{proposition}
\label{set-func}
    Let $\{K_i\}, 1 \le i \le N,$ be a solution to the set maximization problem. Then, for every admissible tuple $\{V_i\}$, one has
    $$
\mathcal{BS}_{\alpha,m}(V_1,\cdots,V_N) \le 
\mathcal{BS}_{\alpha,m}(\tau_1\|x_1\|^{\beta_1}_{K_1},\cdots,\tau_N\|x_N\|^{\beta_N}_{K_N}), 
    $$
    where $\| \cdot\|_{K_i}$ is the Minkowski functional of $K_i$ and $\tau_i = \frac{\frac{1}{\alpha_i}}{\sum_{j=1}^N\frac{1}{\alpha_j}}$.
\end{proposition}
\begin{proof}

By the ``layer-cake'' representation, Prekopa--Leindler, and H\"older inequalities,
one has
\begin{align*}
\bigl(     \prod_{i=1}^{N} \int e^{-\alpha_i V_i(x_i)} dm_i \bigr)^{\frac{1}{\alpha_i}}
& = 
\bigl(     \prod_{i=1}^{N} \int_0^{\infty} m_i(\{x_i : \alpha_i V_i(x_i) \le t_i\}) e^{-t_i} dt_i \bigr)^{\tau_i \sum_{j=1}^N \frac{1}{\alpha_j}}
\\& \le \bigl[ 
\int_0^{\infty}  \sup_{t = \sum_i \tau_i t_i} \Bigl( \prod_{i=1}^{N} m_i( \{ x_i : \alpha_i V_i(x_i) \le t_i\})  \Bigr)^{ \tau_i} e^{-t} dt\bigr]^{\sum_{j=1}^N \frac{1}{\alpha_j}}
\\& \le \bigl[ 
\int_0^{\infty}  \sup_{t = \sum_i \tau_i t_i} \Bigl( \prod_{i=1}^{N} m^{\frac{1}{\alpha_i}}_i( \tilde{A}_i)  \Bigr)^{ \frac{1}{\sum_{j=1}^N \frac{1}{\alpha_j}}} e^{-t} dt\bigr]^{\sum_{j=1}^N \frac{1}{\alpha_j}},
\end{align*}
where
$
\tilde{A}_i = \bigl\{x_i:  V_i(x_i) \le \frac{t_i}{\alpha_i}\bigr\}.
$
One has
$
c(x) \le \sum_i V_i(x_i) \le \sum_{i} \frac{t_i}{\alpha_i}
$
for all $\{x_i\}$ such that $x_i \in \tilde{A}_i$.
Set: 
$$
A_i = \frac{1}{( \sum_{j=1}^N \frac{t_j}{\alpha_j})^{\frac{1}{p}}} \tilde{A}_i.
$$
Using $p$-homogeneity of $c$, one gets $c(x)\le 1
$
provided $x_i \in {A}_i$.
Applying inequality
$
\prod_{i=1}^N  m_i(A_i)^{\frac{1}{\alpha_i}} \le \prod_{i=1}^Nm_i(K_i)^{\frac{1}{\alpha_i}}
$, we obtain the following:
\begin{align*}
 \prod_{i=1}^N m_i(\tilde{A}_i)^{\frac{1}{\alpha_i}} &  = \bigl(\sum_{j=1}^N \frac{t_j}{\alpha_j}\bigr)^{\frac{1}{p}\sum_{i} \frac{n+r_i}{\alpha_i}}  \prod_{i=1}^{N} m_i^{\frac{1}{\alpha_i}}({A}_i)
\le \bigl(\sum_{j=1}^N \frac{t_j}{\alpha_j}\bigr)^{\frac{1}{p}\sum_{i} \frac{n+r_i}{\alpha_i}}  \prod_{i=1}^{N} m_i^{\frac{1}{\alpha_i}}({K}_i)\\& = 
 \bigl(\sum_{j=1}^N \frac{t_j}{\alpha_j}\bigr)^{\sum_{i} \frac{n+r_i}{\alpha_i \beta_i}}  \prod_{i=1}^{N} m_i^{\frac{1}{\alpha_i}}({K}_i) = \prod_{i=1}^N    m_i^{\frac{1}{\alpha_i}}\Bigl( \bigl( \sum_{j=1}^N \frac{t_j}{\alpha_j}\bigr)^{\frac{1}{\beta_i}}K_i\Bigr).
\end{align*}

Hence
\begin{align*}
\int_0^{\infty} & \sup_{t = \sum_i \tau_i t_i} \bigl( \prod_{i=1}^{N} m^{\frac{1}{\alpha_i}}_i( \tilde{A}_i)  \bigr)^{ \frac{1}{\sum_{j=1}^N \frac{1}{\alpha_j}}} e^{-t} dt
 \le 
\int_0^{\infty}  \bigl( \prod_{i=1}^{N} m^{\frac{1}{\alpha_i}}_i\Bigl( \bigl[t \sum_{j=1}^N \frac{1}{\alpha_j}\bigr]^{\frac{1}{\beta_i}} K_i\Bigr)  \bigr)^{ \frac{1}{\sum_{j=1}^N \frac{1}{\alpha_j}}} e^{-t} dt
\\& 
\le \prod_{i=1}^{N} \bigl (\int_0^{\infty}  m_i\Bigl( \bigl[t \sum_{j=1}^N \frac{1}{\alpha_j}\bigr]^{\frac{1}{\beta_i}} K_i\Bigr)  e^{-t} dt \bigr)^{\tau_i}  
= \prod_{i=1}^{N} \bigl (\int_0^{\infty}  m_i\Bigl(
\|x_i\|^{\beta_i}_{K_i} \le t \sum_{j=1}^N \frac{1}{\alpha_j} \Bigr)  e^{-t} dt \bigr)^{ \tau_i} 
\\& = 
\prod_{i=1}^{N} \Bigl( \int e^{-\frac{\|x_i\|^{\beta_i}}{\sum_{j=1}^N \frac{1}{\alpha_j}}} dm_i  \Bigr)^{\tau_i} .
\end{align*}
     Resuming all the inequalities, one gets
$
\bigl(     \prod_{i=1}^{N} \int e^{-\alpha_i V_i(x_i)} dm_i \bigr)^{\frac{1}{\alpha_i}}
\le \prod_{i=1}^{N} \bigl( \int e^{-{\alpha_i \tau_i\|x_i\|^{\beta_i}}} dm_i  \bigr)^{ \frac{1}{\alpha_i}}.
$
The proof is complete.
\end{proof}

\begin{corollary}
\label{100126}
    If $\{\tau_i \|x_i\|^{\beta_i}_{K_i}\}$ is an admissible tuple, i.e.
    $
c(x) \le \sum_{i=1}^N \tau_i \|x_i\|^{\beta_i}_{K_i},
    $
    then it is a maximizing tuple.
\end{corollary}

We observe that, in general, a tuple of functions  $\{\tau_i \|x_i\|^{\beta_i}_{K_i}\}$ may  not be admissible. Hence,  Proposition \ref{set-func} does not imply in general that  $\{\tau_i \|x_i\|^{\beta_i}_{K_i}\}$ is a maximizing tuple.
We give below some sufficient conditions for a tuple  to be admissible.

\begin{theorem} 
\label{30012016}
Let $c$ satisfy all the assumptions from the beginning of the section.
Assume, in addition, that 
 $c$ is $(p_1,\cdots,p_N)$-homogeneous, $$\alpha_i = \frac{1}{p_i} \bigl( 1 + \frac{r_i}{n}\bigr), \ \ \ \ \ \\ \beta_i = \bigl( 1 + \frac{r_i}{n} \bigr) \sum_{j=1}^N \frac{1}{\alpha_j}  =\frac{p_i}{\tau_i}.$$ 
 Assume that  $K_1, \cdots, K_N$ is a maximizing tuple for $\mathcal BS^{set}_{m,\alpha}$. Then $\{\tau_i \|x_i\|^{\beta_i}_{K_i}\}$ is admissible. Hence $\{\tau_i \|x_i\|^{\beta_i}_{K_i}\}$ is  a maximum point of 
$\mathcal{BS}_{\alpha,m}$ on the set of even convex functions.
\end{theorem}
\begin{proof}
According to Lemma \ref{100126}, it is sufficient to prove  the admissibility of $\{\tau_i \|x_i\|^{\beta_i}_{K_i}\}$.
    Take an arbitrary point $x$ such that $x_i \ne 0, 1 \le i \le N$. Then $\frac{x_i}{\|x_i\|_{K_i}} \in K_i$. Hence
    $$
c\bigl(\frac{x_1}{\|x_1\|_{K_1}}, \cdots, \frac{x_N}{\|x_N\|_{K_N}}\bigr) \le 1.
    $$
    Using $p_i$-homogeneity and the  H{\"o}lder inequality, one gets
    $$
c(x) \le \|x_1\|^{p_1}_{K_1} \cdots \|x_N\|^{p_N}_{K_N}
\le \sum_{i=1}^N \frac{p_i}{\beta_i} \|x_i\|^{\beta_i}_{K_i} = \sum_{i=1}^N \tau_i \|x_i\|^{\beta_i}_{K_i}.
    $$
\end{proof}

Theorem \ref{30012016} and Corollary 
\ref{maincorollary} imply the following result.

\begin{corollary}
\label{1002+}    Let all the assumptions of Corollary \ref{maincorollary} be fulfilled. 
    Then every maximizing tuple $\{\Phi_i\}$ of the functional $\mathcal{BS}_{\alpha,m}$  with $\Phi_i(0)=0$ has the form $\{\tau_i \|x_i\|^{\beta_i}_{K_i}\}$, where $\{K_i\}$ is a  maximizing tuple of the functional $\mathcal{BS}^{set}_{\alpha,m}$.
\end{corollary}

 \begin{remark}
     Let all the assumptions of Corollary \ref{maincorollary} be fulfilled. Then the equivalent functional problem on $\mathbb{S} ^{n-1}$ is described in Proposition \ref{1002}.
 \end{remark}

\section{Symmetrization}

In this section, we consider  a generalized  symmetrization procedure for $N>2$ functions  possessing the property to increase the  Blaschke--Santal{\'o} functional. 

\begin{remark}
    Unlike sections 3-4 we {\bf don't assume} homogeneity of the cost function $c$ and the reference measures.
\end{remark}

Our approach closely follows the arguments of Kalantzopoulos and Saroglou from \cite{KS}. In particular, our important $(j,i_1,i_2)$-assumption below
is satisfied by the cost functions considered in \cite{KS}.
However, we establish our result in a  more general situation. In particular, we consider a more general class of functions and the reference measures. 

Let $c(x_1,\cdots,x_N)$ be a continuous function of $N$ variables $x_i \in \mathbb{R}^n$.
As everywhere in this paper, we say that a tuple of  convex symmetric sets $\{A_i\}, 1\le i \le N, A_i \subset \mathbb{R}^n$ is admissible if
$$
c(x_1,\cdots,x_N)\le 1
$$
provided $x_i \in A_i$ for all $1 \le i \le N$.

Let us describe the symmetrization with respect to a coordinate unit vector $e_j$, $1\le j\le n$, for a given couple of indices $i_1 \ne i_2$, $1 \le i_1,i_2 \le N$. 
We will call this procedure 
$(j,i_1,i_2)$-symmetrization.
The procedure will replace sets
$A_{i_1}, A_{i_2}$ with some sets $B_{i_1}, B_{i_2}$ and will not change all the other sets $A_i$.

Given a coordinate unit vector $e_j$, we represent an arbitrary 
$x_i \in \mathbb{R}^n$ in the form
$$
 x_i = (y_i,t_i) = y_i + t_ie_j, \ \  \ y_i \bot e_j.
$$

\begin{definition} 
\label{2402}
    We say that the cost function  $c$ satisfies {\bf the $(j,i_1,i_2)$-assumption}  if  the following holds:
    \begin{enumerate}
    \item 
    The polar $c$-transform preserves symmetric and convex sets, i.e.,
$$
B_j = \{x_j: c(x_1,\cdots,x_N) \le 1, \forall x_i \in B_i, i \ne j\}
$$
is symmetric (convex) if $B_i$ are symmetric (convex) for all $j \ne i$. 
Recall that this is true if $c$ satisfies assumption (2) of Corollary \ref{maincorollary}.

        \item 
    The function
    $$
c_{i_1, i_2}(y_{i_1}, t_{i_2}) = c((y_1,t_1),\cdots, (y_N,t_N))
    $$
    has convex sublevel sets for all fixed $y_{i_k}$, $k \ne 1$, $t_{i_j}$, $j \ne 2$ 
    \item
    \begin{align*}
c((y_1,t_1),  & \cdots, (y_{i_1}, -t_{i_1}),\cdots, (y_{i_2}, t_{i_2}), \cdots (y_N,t_N))
\\& =
c((y_1,t_1), \cdots, (y_{i_1}, t_{i_1}),\cdots, (y_{i_2}, -t_{i_2}), \cdots (y_N,t_N)).
    \end{align*}
    \end{enumerate}
\end{definition}

In what follows, we write
$$
x_i = (x_{i,1}, \cdots, x_{i,n}) = (x_{i,j}), 1 \le j \le n,
$$
for every $1 \le i \le N$.

\begin{example}
    \label{maincost}
    The function 
$$
c_n(x_1,\cdots,x_N) 
= \sum_{j=1}^n x_{1,j}x_{2,j} \cdots x_{N,j} 
=  \sum_{j=1}^n \prod_{i=1}^N x_{i,j}
$$
admits the following property:
$
c_n((y_1,t_1),\cdots, (y_N,t_N))
= c_{n-1}(y_1,\cdots, y_N)
+ \prod_{i=1}^N t_i.
$
Clearly, $c$ satisfies the $(j,i_1,i_2)$-assumption for all $j, i_1,i_2$.

In addition, the function $\sum_{j=1}^n \prod_{i=1}^N |x_{i,j}|^{p_{ij}}$ satisfies   the $(j,i_1,i_2)$-assumption for all $j, i_1,i_2$ if $p_{i,j} \ge 1$ for all $i,j$.
\end{example}

Recall that given a convex set $A$, the set $\tilde{A}$ obtained by the Steiner symmetrization with respect to $e_j$ can be represented as follows:
$$
\tilde{A} = \bigl\{\big(y,\frac{t-s}{2}\big), (y,t) , (y,s) \in A\bigr\}.
$$

\begin{lemma}
\label{sym}
    Let $\{A_i\}$ be an admissible tuple of symmetric convex bodies, where the cost function $c$ satisfies the $(n,1,2)$-assumption and the polar $c$-transform preserves symmetric convex sets.

      Let $m_1 = \rho_1 dx_1, m_2 =\rho_2 dx_2$ be locally finite measures admitting the following properties:
\begin{enumerate}
    \item 
    $t \to \rho_1(y + te_n)$ is an even function, decreasing on $[0,+\infty)$ for
    all $y \bot e_n$
    \item 
    $
 y \to \rho_2( y+te_n)
    $
    is a symmetric log-concave function for all $t \in \mathbb{R}$.
\end{enumerate}
      
    Then there exists an admissible tuple of the form
    $
\{\tilde{A}_1, B,A_3,\cdots,A_{N}\}
    $
    such that
    $$
m_1(\tilde{A}_1) \ge m_1(A_1), \ \ \
m_2(B) \ge m_2(A_2).
    $$
\end{lemma}
\begin{proof}
We observe first that the Steiner symmetrization $\tilde{A}_1$ of $A_1$ satisfies $m_1(\tilde{A}_1) \ge m_1(A_1)$. This is a known property (see, for instance, Proposition 4.7 in \cite{CKLR} for the functional version of the statement and the references therein).

    Without loss of generality, one can take for $A_2$ the polar $c$-transform of $A_1,  A_3, \cdots, A_N$:
    $$
A_2 = \{ x_2: c(x) \le 1, \mbox{for all } x_i \in A_i,  i \ne 2\}.
    $$
    Similarly, we define $B$ as the $c$-transform of  $\tilde{A}_1$ and $A_i$ with $i>2$:
    $$
B = \{ x_2: c(x) \le 1, \mbox{for all }x_1 \in \tilde{A}_1, x_i \in A_i,  i > 2\}.
    $$
Since the $c$-polar transform preserves symmetric convex bodies, both bodies are symmetric and convex.

      For an arbitrary $A \subset \mathbb{R}^n$, we denote by $A(r)$  the section of $A$ at the level $r$:
    $$
A(r) = \{y: (y,r) = y + r e_n \in A\}.
    $$
    We want to prove that
    \begin{equation}
    \label{aver}
\frac{A_2(r) + A_2(-r)}{2} \subset B(r).
    \end{equation}
    Next, we use  (\ref{aver}), the  Prekopa--Leindler inequality, the assumptions of symmetry of $A_2$, and the log-concavity of $\rho_2(y+te)$: 
\begin{align*}
    m_2(B) & = \int_{\mathbb{R}}
 \bigl( \int_{B(r)} \rho_2(y+te_n)dy \bigr) dt
 \ge \int_{\mathbb{R}}   
 \bigl( \int_{\frac{A_2(r) + A_2(-r)}{2}} \rho_2(y+te_n)dy \bigr) dt
 \\& \ge 
  \int_{\mathbb{R}}   
 \bigl( \int_{A_2(r)} \rho_2(y+te_n)dy \bigr)^{\frac{1}{2}} \bigl( \int_{A_2(-r)} \rho_2(y+te_n)dy \bigr)^{\frac{1}{2}} dt
 = \int_{\mathbb{R}}   
  \int_{A_2(r)} \rho_2(y+t e_n)dy  dt \\& = m_2(A_2).
\end{align*}
    It remains to prove (\ref{aver}). One has
    \begin{align*}
B(r)  = & \{v : c\bigl((y_1,(t-s)/{2}),(v,r), (y_3,t_3),\cdots,(y_N,t_N) \bigr) \le 1, \\& \mbox{for all } (y_1,t) \in A_1, (y_1,s) \in A_1, (y_i,t_i) \in A_i, i >2 \}.
    \end{align*}
Take an arbitrary $a \in A_2(r),b \in A_2(-r)$. One has
$$
c\bigl((y_1,t),(a,r), (y_3,t_3),\cdots,(y_N,t_N)\bigr) \le 1, 
\ \ \ 
c\bigl((y_1,s),(b,-r), (y_3,t_3),\cdots,(y_N,t_N)\bigr) \le 1, 
$$
for all $(y_1,t),(y_1,s)\in A_1, (y_i,t_i) \in A_i, i >2$.
Hence,
$$
c\bigl((y_1,-s),(b,r), (y_3,t_3),\cdots,(y_N,t_N)
= c\bigl((y_1,s),(b,-r), (y_3,t_3),\cdots,(y_N,t_N)
\le 1,
$$
by  item (3) of the $(n,1,2)$ assumption (see Definition \ref{2402}), hence by  item (2) 
$$
1 \ge c\bigl((y_1,(t-s)/{2}),((a+b)/{2},r), (y_3,t_3),\cdots,(y_N,t_N)\bigr).
$$
Thus $\frac{a+b}{2} \in B(r)$ and the proof is completed.    
\end{proof}

Recall that a set $A$ is called unconditional with respect  to the basis $e_1, \cdots,e_n$ if 
$$
x = \sum_{i=1}^n x_i e_i \in A \Longleftrightarrow \sum_{i=1}^n s_ix_i e_i \in A
$$
for any tuple $\{s_i\}$ such that $s_i \in \{-1,1\}$.

\begin{corollary}
\label{uncond-reduction}
    Let $\{m_i\}$ be log-concave unconditional measures, $1 \le i \le N$, and $c$ satisfy the $(j,i_1,i_2)$-assumption for all $j,i_1,i_2$. 

Assume, in addition, that the polar $c$-transform preserves 1) symmetric sets, 2) unconditional sets, 3) convex sets. 
Then, for every admissible tuple $\{A_i\}$ of symmetric convex bodies, there exists an admissible tuple $\{B_i\}$ of convex unconditional bodies with the property 
    $$
    m_i(A_i) \le m_i(B_i).
    $$
\end{corollary}
\begin{proof}
    We consequently apply the symmetrization with respect to $e_1, \cdots,e_n$ to the set $A_1$, keeping $A_3, \cdots, A_N$ fixed and defining the second body in the tuple by the $c$-polar transform. By Lemma  \ref{sym}, one finally obtains another tuple of even convex sets $\{A'_i\}$ such that $m_i(A'_i) \ge m_i(A_i)$, and, in addition, $A'_1$ is unconditional. Repeating this procedure $N-1$ times, we get unconditional sets $B_1, \cdots, B_{N-1}$. The set $B_N$ will be the $c$-polar transform of  $B_1, \cdots, B_{N-1}$. Since this operation preserves unconditional sets, we get the claim.
\end{proof}

\section{Main example}

In this section, we discuss examples of the Blaschke--Santal{\'o}-type inequalities, generalizing Theorem \ref{example1} from \cite{KS}.
Our extensions go in the following directions:
1) extension to unconditional log-concave reference measures in the spirit of Theorem 7 from \cite{FMZ},
2) extension to product functionals with different weights, i.e., functionals of the type $\prod_{i=1}^N m^{\frac{1}{\alpha_i}}(A_i)$.

For the case $N=2$, the following result was obtained in Fradelizi--Meyer--Zvavitch \cite{FMZ} (Theorem 7): 
 \begin{equation}
 \label{BS-CE}
\mu(A) \mu(A^{\circ}) \le \mu(B_2)^2
 \end{equation}
for any convex symmetric $A$ and unconditional log-concave measure $\mu$.  Cordero-Erausquin conjectured in \cite{CEcomplex} that (\ref{BS-CE}) is true for any symmetric log-concave $\mu$.

\begin{proposition}
\label{1211}
Let $f_i \colon \mathbb{R}_+^n \to \mathbb{R}_+$, $ 1\le i \le N$ be Borel functions,
$\{\alpha_i\}$ be $N$ positive numbers, and
$$
c(x_1,\cdots,x_N) 
= \sum_{j=1}^n x^{\frac{1}{\alpha_1}}_{1,j} \cdots x^{\frac{1}{\alpha_N}}_{N,j}.
$$
Assume that the  functions $\{f_i(x_i)\}$ satisfy the inequality
$$
\prod_{i=1}^N f^{\frac{1}{\alpha_i}}_i(x_i)
\le \rho(c(x))
$$
for some Borel function $\rho$. 
Let $m=e^{-V}dx$ be a measure on $\mathbb{R}^n_+$ such that the function $(t_1, \cdots,t_n) \to V(e^{t_1}, \cdots, e^{t_n})$ is convex.

Then 
\begin{align*}
 \prod_{i=1}^N \Bigl[ 
 \int_{\mathbb{R}^n_+} f_i(x_i)   m(dx_i)\Bigr]^{\frac{1}{\alpha_i}}
\le \Bigl( \int_{\mathbb{R}^n_+} \rho^{\frac{1}{A}}\bigl(\sum_{j=1}^n t^A_j\bigr) m(dt)\Bigr)^A,
\end{align*}
where $A =\sum_{j=1}^N \frac{1}{\alpha_j}$.

In particular, if a tuple of functions $\{V_i(x_i)\}$ satisfies $\sum_{i=1}^N V_i(x_i) \ge \sum_{j=1}^n x^{\frac{1}{\alpha_1}}_{1,j} \cdots x^{\frac{1}{\alpha_N}}_{N,j}$, then
$$
\prod_{i=1}^N \Bigl[ 
 \int_{\mathbb{R}^n_+} e^{-\alpha_i V_i(x_i)} m(dx_i)\Bigr]^{\frac{1}{\alpha_i}}
\le \Bigl( \int_{\mathbb{R}^n_+} e^{-\frac{1}{A}\sum_{j=1}^n t^A_j} m(dt)\Bigr)^A.
$$
\end{proposition}
\begin{proof} Following \cite{FradeliziMeyer2008(2)} we apply the exponential change of variables :
$x_{i,j} = e^{t_{i,j}}$ (or simply $x_i = e^{t_i}$) and the Prekopa--Leindler inequality
\begin{align*}
 \prod_{i=1}^N \bigl[ 
 \int_{\mathbb{R}^n} f_i(x_i)   e^{-V(x_i)}dx_i\bigr]^{\frac{1}{\alpha_i}} & =      \prod_{i=1}^N \bigl[ 
 \int_{\mathbb{R}^n} f_i(x_i) e^{-V(x_i)}   dx_{i,1} \cdots dx_{i,n}\bigr]^{\frac{1}{\alpha_i}}
\\& =
 \prod_{i=1}^N \bigl[ 
 \int_{\mathbb{R}^n_+} f_i(e^{t_i}) e^{-V(e^{t_i})}e^{ \sum_{j=1}^n t_{i,j}}  dt_{i,1} \cdots dt_{i,n}\bigr]^{\frac{1}{\alpha_i}}
 \\& \le \bigl( \int_{\mathbb{R}^n} \sup_{s_k = \frac{1}{A} \sum_{i=1}^N \frac{t_{i,k}}{\alpha_i} } 
 \prod_{i=1}^N f_i^{{\frac{1}{A\alpha_i}}} (e^{t_i}) e^{-\frac{1}{A\alpha_i}V(e^{t_i})}
e^{\sum_{k=1}^n s_k}  ds_1 \cdots ds_n\bigr)^A
\\& \le \bigl( \int_{\mathbb{R}^n}\sup_{s_k = \frac{1}{A} \sum_{i=1}^N \frac{t_{i,k}}{\alpha_i} } 
\rho^{\frac{1}{A}}\Bigl(\sum_{j=1}^n e^{\sum_{i=1}^N\frac{t_{i,j}}{\alpha_j}}\bigr) e^{-V(e^{\frac{1}{A} \sum_{i=1}^N \frac{t_i}{\alpha_i}})}
e^{\sum_{k=1}^n s_k}  ds\bigr)^A
\\& = \bigl( \int_{\mathbb{R}^n} \rho^{\frac{1}{A}}\bigl(\sum_{j=1}^n e^{As_j}\bigr)
e^{\sum_{k=1}^n s_k}   e^{-V(e^s)} ds\bigr)^A
= \bigl( \int_{\mathbb{R}^n_+} \rho^{\frac{1}{A}}\bigl(\sum_{j=1}^n t^A_j\bigr) e^{-V(t)} dt\bigr)^A.
\end{align*}
\end{proof}

\begin{theorem}
(Extension of Theorem \ref{example1})
Let  $N \ge 2$ be a natural number and $m$ be a log-concave  unconditional measure on $\mathbb{R}^n$. Consider the following cost function: $$
c(x) = c(x_1,\cdots,x_N) 
= \sum_{j=1}^n x_{1,j}x_{2,j} \cdots x_{N,j} 
$$ 
on $(\mathbb{R}^n)^N$.
Assume we are given $N$ symmetric convex sets
 $A_i \subset \mathbb{R}^n$, $1 \le i \le N$  such that
$c(x_1,\cdots,x_N) \le 1$ on the set $\prod_{i=1}^N A_i \subset (\mathbb{R}^n)^N$. 
Then
$$
\prod_{i=1}^N m(A_i) \le m(B_N)^N.
$$
Let $\{V_i\}$, $1 \le i \le N,$ be even measurable functions with values in $(-\infty,+\infty]$ and satisfying
$
\sum_{i=1}^N V_i(x_i) \ge c(x).
$
Assume, in addition, that $e^{-V_i}\in L^1(m)$ for every $i$. 
Then
$$
\prod_{i=1}^N \int_{\mathbb{R}^n} e^{-V_i(x_i)} m(dx_i)
\le \Bigl( \int_{\mathbb{R}^n} e^{-\frac{1}{N} \sum_{i=1}^n |y_i|^N} m(dy)\Bigr)^N.
$$
In addition,  let $m$ be the Lebesgue measure and $\{\Phi_i\}$ be a maximizing tuple of the functional $\prod_{i=1}^N \int_{\mathbb{R}^n} e^{-V_i(x_i)} dx_i $ under the same constraints with $\Phi_i(0)=0$.Then $\Phi_i = \frac{1}{N} \|x_i\|^N_{K_i}$, where $\{K_i\}$ is a maximizing tuple for the corresponding set problem and $\|\cdot\|_{K_i}$ are the associated Minkowski functionals.
\end{theorem}
\begin{proof}
Clearly, the $c$-polar transform preserves symmetric and unconditional  sets, and according to Example \ref{maincost}, the cost function $c$
satisfies the $(j,i_1,i_2)$-assumption for all indices.
By Corollary \ref{uncond-reduction}, the desired inequality can be reduced to the unconditional case. Thus, without loss of generality, one can assume that the sets $A_i$ are unconditional. The same symmetrization procedure can be applied to functions.  We don't present the complete proof here, but we can refer to \cite{CKLR} Theorem 4.1, where the symmetrization for functions was considered in detail for $N=2$. The case $N>2$ follows the arguments 
of Lemma \ref{sym} and Theorem 4.1 of \cite{CKLR}.
Finally, after the reduction to the unconditional case, both inequalities (for sets and functions) follow from Proposition \ref{1211}.
The statement about the homogeneity of the maximizers follows from Corollary \ref{maincorollary}. Finally, the description of the maximizers in terms of the Minkowski functionals follows from Corollary \ref{1002+}.
\end{proof}

The same arguments can be applied in the proof of  the following result.

\begin{theorem}
\label{example2}
Consider the cost function  $$
c(x_1,\cdots,x_N) 
= \sum_{j=1}^n |x_{1,j}|^{\frac{1}{\alpha_1}} |x_{2,j}|^{\frac{1}{\alpha_2}} \cdots |x_{N,j}|^{\frac{1}{\alpha_N}}, 
$$ where $\alpha_i \le 1$.
Let $A_i \subset \mathbb{R}^n$, $1 \le i \le N$ be $N$ symmetric convex sets satisfying
$c(x_1,\cdots,x_N) \le 1$, provided $x_i \in A_i$, $1 \le i \le N$; and let $m$ be a log-concave unconditional measure. Then
$$
\prod_{i=1}^N m(A_i)^{\frac{1}{\alpha_i}} \le m(B)^{A}, 
$$
where $B = \bigl\{ t:  \sum_{j=1}^n |t_j|^{A} \le 1\bigr\}$, $A = \sum_{i=1}^N \frac{1}{\alpha_i}$.

Let $\{V_i\}$ be even measurable functions with values in $(-\infty,+\infty]$ and satisfying
$
\sum_{i=1}^N V_i(x_i) \ge c(x).
$ 
Assume, in addition, that $e^{-V_i}\in L^1(m)$ for every $i$. 
Then
$$
\prod_{i=1}^N \Bigl(\int_{\mathbb{R}^n} e^{-\alpha_iV_i(x_i)} m(dx_i) \Bigr)^{\frac{1}{\alpha_i}}
\le \Bigl( \int_{\mathbb{R}^n} e^{-\frac{1}{A} \sum_{i=1}^n |y_i|^A} m(dy)\Bigr)^A.
$$
In addition, let $m$ be the Lebesgue measure. Then every maximizing tuple $\{\Phi_i\}$ with $\Phi_i(0)=0$ consists of $A$-homogeneous functions of the form $\{{\frac{1}{A\alpha_i}} \|x_i\|^A_{K_i}\}$, where $\{K_i\}$ is a maximizing tuple of the corresponding set problem.
\end{theorem}

\begin{remark}
    In view of Theorems \ref{example1}, \ref{example2}, it is natural to ask about examples of the Blaschke--Santal{\'o} inequality for homogeneous non-Lebesgue reference measures. One such example can be constructed from the result of Theorem \ref{example2} with the Lebesgue reference measure just by the change of variables
$$
x_{i,j} = {\rm{sign}}(y_{i,j}) |y_{i,j}|^{\tau_i}, \ \ \ {\tau_i} \ne 0.
$$
Examples of maximizers for the functional 
\begin{equation}
    \label{BSnongauss}
|A|^{\frac{1}{p}} \Bigl(\int_{A^\circ} \rho dy \Bigr)^{\frac{1}{q}},
\end{equation}
where $p >1 , \frac{1}{p} + \frac{1}{q}=1$, and $\rho$ is $\frac{2-p}{p-1}$-homogeneous, can be found in \cite{CKLR}.
We observe that it is problematic to extend  the inequality (\ref{BSnongauss}) to the case of $N>2$ sets using the symmetrization in the homogeneous case. This is  because we need the log-concavity of measures for the symmetrization procedure,  and  this assumption does not fit well with the homogeneity. 
\end{remark}

\end{document}